\documentclass[10pt,a4paper]{article}
\usepackage[a4paper]{geometry}
\usepackage{amssymb,latexsym,amsmath,amsfonts,amsthm}
\usepackage{graphicx,comment}
\usepackage{epsfig}
\usepackage{tikz}
\usepackage[T1]{fontenc}

\newcommand{\lf}{\left\lfloor}
\newcommand{\rf}{\right\rfloor}
\newcommand{\lc}{\left\lceil}
\newcommand{\rc}{\right\rceil}

\newcommand{\Fkr}{\widetilde{F}_{k}^{(\rho)}}
\newcommand{\Fk}{F_{k}^{(\rho)}}
\newcommand{\Fkmr}{\widetilde{F}_{k-1}^{(\rho)}}
\newcommand{\Fkm}{F_{k-1}^{(\rho)}}
\newcommand{\Fkpr}{\widetilde{F}_{k+1}^{(\rho)}}
\newcommand{\Fkp}{F_{k+1}^{(\rho)}}
\newcommand{\Skr}{S_{k}^{(\rho)}}

\newtheorem{theorem}{Theorem}[section]
\newtheorem{lemma}[theorem]{Lemma}
\newtheorem{proposition}[theorem]{Proposition}

\newtheorem{corollary}[theorem]{Corollary}

\theoremstyle{definition}
\newtheorem{definition}[theorem]{Definition}

\theoremstyle{remark}

\newtheorem{remark}[theorem]{Remark}

\numberwithin{equation}{section}

\title{Nikishin systems on star-like sets: Ratio asymptotics of the associated multiple orthogonal polynomials}

\date{\today}

\author{Abey L\'{o}pez-Garc\'{i}a\footnotemark[1] \quad Guillermo L\'{o}pez Lagomasino\footnotemark[2]}

\begin{document}

\maketitle

\par{\centering Dedicated to the memory of our dear friend Valery V. Vavilov \par}

\renewcommand{\thefootnote}{\fnsymbol{footnote}}
\footnotetext[1]{Department of Mathematics and Statistics, University of South Alabama, 411 University Blvd North, Mobile AL, 36688, USA. Current address: Department of Mathematics, University of Central Florida, 4393 Andromeda Loop N, Orlando, FL 32816, USA. email: Abey.Lopez-Garcia\symbol{'100}ucf.edu.} \footnotetext[2]{Departamento de
Matem\'{a}ticas, Universidad Carlos III de Madrid, Avda.~Universidad 30, 28911 Legan\'{e}s, Madrid, Spain. email:
lago\symbol{'100}math.uc3m.es.\\ \indent Both authors were partially supported by the grant MTM2015-65888-C4-2-P of the Spanish Ministry of Economy and Competitiveness.
}

\abstract{We investigate the ratio asymptotic behavior of the sequence $(Q_{n})_{n=0}^{\infty}$ of multiple orthogonal polynomials associated with a Nikishin system of $p\geq 1$ measures that are compactly supported on the star-like set of $p+1$ rays $S_{+}=\{z\in\mathbb{C}: z^{p+1}\geq 0\}$. The main algebraic property of these polynomials is that they satisfy a three-term recurrence relation of the form $zQ_{n}(z)=Q_{n+1}(z)+a_{n} Q_{n-p}(z)$ with $a_{n}>0$ for all $n\geq p$. Under a Rakhmanov-type condition on the measures generating the Nikishin system, we prove that the sequence of ratios $Q_{n+1}(z)/Q_{n}(z)$ and the sequence $a_{n}$ of recurrence coefficients are limit periodic with period $p(p+1)$. Our results complement some results obtained by the first author and Mi\~{n}a-D\'{i}az in a recent paper in which algebraic properties and weak asymptotics of these polynomials were investigated. Our results also extend some results obtained by the first author in the case $p=2$.
\smallskip

\textbf{Keywords:} Multiple orthogonal polynomial, Nikishin system, banded
Hessenberg matrix, ratio asymptotics, interlacing of zeros.\smallskip

\textbf{MSC 2010:} Primary $42C05$, $30E10$; Secondary $47B39$.}

\tableofcontents

\section{Introduction}

This paper is a continuation of the investigations initiated in \cite{LopMin, Lop} on properties of multiple orthogonal polynomials associated with Nikishin systems of measures supported on star-like sets. Nikishin systems were introduced by Nikishin himself in his seminal work \cite{Nik}, which served as a starting point for a prolific study of the associated families of multiple orthogonal polynomials and Hermite-Pad\'{e} approximants.

There is now a rather comprehensive literature on the theory of multiple orthogonal polynomials associated with Nikishin systems \emph{on the real line}, which includes the so-called type I, type II, and mixed-type polynomials. Among the topics that have been investigated within this theory we find, e.g., strong asymptotics \cite{Apt}, ratio asymptotics \cite{AptLopRocha,LopLopratio,FidLopLopSor,FidLop}, relative asymptotics \cite{LopLoprel}, zero asymptotic distribution \cite{NikSor, GonRakSor, FidLopLopSor}, Hermite-Pad\'{e} approximation \cite{Nik,Nik2,NikSor,BusLop,DriStI,DriStII,GonRakSor, FidLopHP,LopMed}, recurrence relations \cite{AptKalLopRocha,DelLopLop}, normality and perfectness \cite{DriStnorm,FidLop}, and the list could be enlarged.

Recently, the study of Nikishin systems \emph{on star-like sets} has been motivated by the study of sequences of polynomials $(Q_{n})_{n=0}^{\infty}$ that satisfy a high order three-term recurrence relation of the form
\begin{equation}\label{hthreeterm}
z Q_{n}(z)=Q_{n+1}(z)+a_{n} Q_{n-p}(z),\qquad a_{n}>0,\quad n\geq p,
\end{equation}
where $p$ is a fixed positive integer. Early works that have investigated such recursions are those of Eiermann-Varga \cite{EierVarga} and He-Saff \cite{HeSaff} on Faber polynomials associated with hypocycloidal domains (the constant coefficient case $a_{n}=1/p$, $n\geq p$). Later, Aptekarev-Kalyagin-Van Iseghem \cite{AptKalIseg} studied \eqref{hthreeterm} under no additional hypotheses. They proved a Favard-type theorem, showing that the polynomials $(Q_{n})_{n=0}^{\infty}$ satisfying \eqref{hthreeterm} with initial conditions
\begin{equation}\label{eq:initcondrr}
Q_{\ell}(z)=z^{\ell},\qquad 0\leq \ell\leq p,
\end{equation}
are multi-orthogonal (in the same non-Hermitian sense of Definition~\ref{def:MOP} below) with respect to a system of $p$ complex measures $\mu_{0},\ldots,\mu_{p-1}$ supported on the star-like set
\[
S_{+}=\{z\in\mathbb{C}: z^{p+1}\geq 0\}.
\]
The collection $\{\mu_{0},\ldots,\mu_{p-1}\}$ can be regarded as the system of \emph{spectral measures} \cite{AptKalIseg,Kal1,Kal2} of the difference operator given in the standard basis of $l^{2}(\mathbb{N})$ by the infinite $(p+2)$-banded Hessenberg matrix
\begin{equation}\label{operator}\left(
\begin{array}{ccccccc}
0 &1&0&0&0&\ldots &\ldots\\
0 &0&1&0&0&\ldots &\ldots\\
0 &0&0&1&0&\ldots &\ldots\\
\ldots &\ldots&\ldots&\ldots&\ldots&\ldots &\ldots\\
a_p &0&0&0&0&\ldots &\ldots\\
0 &a_{p+1}&0&0&0&\ldots &\ldots\\
0 &0 & a_{p+2}&0&0&\ldots &\ldots\\
\ldots &\ldots&\ldots&\ldots&\ldots&\ldots &\ldots
\end{array}
\right).
\end{equation}

More recently, Aptekarev-Kalyagin-Saff \cite{AptKalSaff} considered strong asymptotics of polynomials $Q_{n}$ satisfying \eqref{hthreeterm}--\eqref{eq:initcondrr} under the hypothesis
\[
\sum_{n=p}^{\infty}|a_{n}-a|<\infty,\qquad a>0,
\]
and some properties of the measures $\mu_{j}$ were also deduced. In particular, for the first time a formal connection with Nikishin systems on star-like sets was established. In \cite{DelLop}, Delvaux and the first author studied in more detail properties of polynomials $Q_{n}$ satisfying \eqref{hthreeterm}--\eqref{eq:initcondrr}, analyzing them in the more general setting of Riemann-Hilbert minors (or generalized eigenvalue polynomials associated with truncations of the matrix \eqref{operator}). A variety of asymptotic and non-asymptotic results were obtained there, and in particular a connection was explicitly established between \eqref{hthreeterm} and Nikishin systems on star-like sets in the case of periodic recurrence coefficients satisfying some additional conditions, see \cite[Theorem 2.10]{DelLop}. Another paper that has recently studied \eqref{hthreeterm} in connection with the location and interlacing of the zeros of the polynomials $Q_{n}$ is Ben Romdhane \cite{Romd}.

The results mentioned so far can be regarded as direct spectral results, since they are obtained under assumptions on the recurrence coefficients. The first paper that analyzed an inverse spectral problem for \eqref{operator} was \cite{Lop}. In that work a Nikishin system consisting of two measures (case $p=2$) was considered on a $3$-star centered at the origin, and the asymptotic behavior of the associated multiple orthogonal polynomials (defined as in Definition~\ref{def:MOP} below) was studied. In particular the ratio asymptotic behavior was described under a Rakhmanov-type condition on the measures generating the Nikishin system, and it was observed that this behavior was limit periodic with period $6$, which was shomewhat of a surprise (in the analogous situation on the real line the period is $2$, cf. \cite{AptLopRocha,AptKalLopRocha}). The main goal of the present work is to obtain a generalization of this result for an arbitrary $p\geq 1$.

Our main reference in this paper will be \cite{LopMin}, where the fundamental algebraic properties of the polynomials under investigation were proved. In \cite{LopMin} the authors also described the zero asymptotic distribution of these polynomials under regularity conditions on the measures generating the Nikishin system.

The strategy that we follow in the present work to obtain our ratio asymptotic results was first used in \cite{AptLopRocha}, where an analogue of Rakhmanov's celebrated theorem on ratio asymptotics \cite{Rakh1,Rakh2} was first proved for multiple orthogonal polynomials associated with Nikishin systems on the real line. That strategy has been applied in several other papers \cite{LopLopratio,FidLopLopSor,FidLop,Lop}. It is based on the analysis of certain boundary value relations between the limiting functions and the application of asymptotic results of orthogonal polynomials on the real line with respect to varying measures. 

This paper is organized as follows. In Section~\ref{section:Pre} we define Nikishin systems on star-like sets, and reproduce those results and definitions from \cite{LopMin} that will be needed for our analysis. Only the information that is strictly relevant has been presented. In Section~\ref{section:Main} we state our main results. In Section~\ref{sec:diffZnk} we analyze the difference $Z(n+1,k)-Z(n,k)$, where the values $Z(n,k)$~\eqref{def:Znk} are the degrees of certain polynomials $P_{n,k}$ introduced in Definition~\ref{Def:polyPnk} that play a leading role in our analysis. In Section~\ref{sec:interlacing} we prove the interlacing property of the zeros of $P_{n,k}$. In the last section we analyze the ratio asymptotic behavior of the polynomials $P_{n,k}$ and prove our main asymptotic results.

\section{Preliminaries}\label{section:Pre}

In this section we describe the background material from \cite{LopMin} that is essential for the present work. We start with the definition of Nikishin systems on stars.

\subsection{Definition of Nikishin system on a star and induced hierarchy on the real line}\label{sect:defNik}

Let $p\geq 1$ be an integer, and let
\[
S_{\pm}:=\{z\in\mathbb{C}: z^{p+1}\in\mathbb{R}_{\pm}\},\qquad \mathbb{R}_{+}=[0,+\infty),\quad \mathbb{R}_{-}=(-\infty,0].
\]

We construct $p$ finite stars contained in $S_{\pm}$ as follows:
\begin{align*}
\Gamma_{j} & :=\{z\in\mathbb{C}: z^{p+1}\in[a_{j},b_{j}]\},\qquad \quad 0\leq
j\leq p-1,
\end{align*}
where
\begin{equation}\label{relpositint}
\begin{aligned}
0\leq a_{j}<b_{j}<\infty,\quad j\equiv 0\mod 2,\\
-\infty< a_{j}<b_{j}\leq 0, \quad j\equiv 1\mod 2.
\end{aligned}
\end{equation}
Hence $\Gamma_{j}\subset S_{+}$ if $j$ is even, and $\Gamma_{j}\subset S_{-}$
if $j$ is odd. We assume throughout that $\Gamma_{j}\cap
\Gamma_{j+1}=\emptyset$ for all $0\leq j\leq p-2$.

We define now a Nikishin system on $(\Gamma_{0},\ldots,\Gamma_{p-1})$. For each
$0\leq j\leq p-1$,  let $\sigma_{j}$ denote a positive, rotationally invariant
measure on $\Gamma_{j}$, with infinitely many points in its support. These will
be the measures generating the Nikishin system.

Let
\[
\widehat{\mu}(x):=\int\frac{d\mu(t)}{x-t}
\]
denote the Cauchy transform of a complex measure $\mu$, and let
$\mu_1,\ldots,\mu_N$ be $N\geq 1$ measures such that $\mu_j$ and $\mu_{j+1}$
have disjoint supports for every $1\leq j\leq N-1$.
We define the measure $\langle \mu_{1},\ldots,\mu_{N}\rangle$ by the following
recursive procedure.
For $N=1$, $\langle \mu_{1}\rangle:= \mu_{1}$, for $N=2$,
\[
d\langle \mu_{1},\mu_{2}\rangle (x):= \widehat{\mu}_{2}(x)\,d\mu_{1}(x),
\]
and for $N>2$,
\[
\langle \mu_{1},\ldots,\mu_{N}\rangle := \langle \mu_{1},\langle
\mu_{2},\ldots,\mu_{N}\rangle\rangle.
\]

We then define the Nikishin system
$(s_{0},\ldots,s_{p-1})=\mathcal{N}(\sigma_0,\ldots,\sigma_{p-1})$ generated by
the vector of
$p$ measures $(\sigma_0,\ldots,\sigma_{p-1})$ by setting
\begin{equation}\label{def:sj}
s_{j}:=\langle \sigma_{0},\ldots, \sigma_{j}\rangle, \qquad 0\leq j\leq p-1.
\end{equation}
Notice that these measures $s_{j}$ are all supported on the first star $\Gamma_{0}$.

An alternative and more convenient way to define the Nikishin system $(s_{0},\ldots,s_{p-1})$ is as the first row of
the following hierarchy  of measures $s_{k,j}$,
\begin{align}\label{hiers}
 \begin{array}{ccccc}
  s_{0,0} & s_{0,1} & s_{0,2} & \cdots & s_{0,p-1}\\
  & s_{1,1} & s_{1,2} & \cdots & s_{1,p-1}\\
& & s_{2,2} & \cdots & s_{2,p-1}\\
 &  & & \ddots & \vdots\\
  &  & &  & s_{p-1,p-1}
 \end{array}
\end{align}
where
\[
 s_{k,j}=\langle \sigma_k,\ldots,\sigma_j\rangle,\quad 0\leq k\leq j\leq p-1.
\]
More descriptively, the measures $s_{k,j}$ are inductively defined by setting
\[
\begin{split}
s_{k,k}:={}&\sigma_k, \quad 0\leq k\leq p-1,\\
d s_{k,j}(z)={}&\int_{\Gamma_{k+1}} \frac{ds_{k+1,j}(t)}{z-t}d\sigma_k(z),\quad
0\leq k < j\leq p-1.
\end{split}
\]

Notice then that for each pair $k$, $j$ with $0\leq k\leq j\leq p-1$,
$(s_{k,k},\ldots,s_{k,j})=\mathcal{N}(\sigma_k,\ldots,\sigma_j)$ is the
Nikishin system generated by
$(\sigma_k,\ldots,\sigma_j)$.

For every $0\leq j\leq p-1$, we shall denote by $\sigma^*_j$ the push-forward of $\sigma_j$ under the map $z\mapsto z^{p+1}$, that is,
 $\sigma^*_j$ is the measure on $[a_j,b_j]$ such that for every Borel set $E\subset [a_j,b_j]$,
\begin{equation}\label{def:sigma:star}
\sigma^*_j(E):=\sigma_j\left(\{z:z^{p+1}\in E\}\right).
\end{equation}

We now construct, out of these $\sigma_j^*$, a new hierarchy of measures
$\mu_{k,j}$, $0\leq k\leq j\leq p-1$:
\begin{align}\label{hiermu}
 \begin{array}{ccccc}
  \mu_{0,0} & \mu_{0,1} & \mu_{0,2} & \cdots & \mu_{0,p-1}\\
  & \mu_{1,1} & \mu_{1,2} & \cdots & \mu_{1,p-1}\\
& & \mu_{2,2} & \cdots & \mu_{2,p-1}\\
 &  & & \ddots & \vdots\\
  &  & &  & \mu_{p-1,p-1}
 \end{array}
\end{align}
where  the measures $\mu_{k,j}$ are inductively defined by setting
\[
\mu_{k,k}:=\sigma^*_k, \quad 0\leq k\leq p-1,
\]
\[
d
\mu_{k,j}(\tau)=\left(\tau\int_{a_{k+1}}^{b_{k+1}}\frac{d\mu_{k+1,j}(s)}{\tau-s}
\right)d\sigma^*_{k}(\tau),\quad \tau \in [a_k,b_k],\quad 0\leq k < j\leq p-1.
\]

The following result, proved in \cite[Prop. 2.2]{LopMin}, describes the
relation between the measures $s_{k,j}$ and $\mu_{k,j}$.

\begin{proposition} \label{proposition2} For every $0\leq k\leq j\leq p-1$, we
have
\begin{align*}
\int_{\Gamma_{k}}\frac{ds_{k,j}(t)}{z-t}={}&
z^{p+k-j}\int_{a_{k}}^{b_{k}}\frac{d\mu_{k,j}(\tau)}{z^{p+1}-\tau},
\end{align*}
that is,
\[
\widehat{s}_{k,j}(z)=z^{p+k-j}\widehat{\mu}_{k,j}(z^{p+1}).
\]
\end{proposition}

\subsection{Multiple orthogonal polynomials and functions of the second kind}

\begin{definition}\label{def:MOP}
Let $\{Q_{n}(z)\}_{n=0}^\infty$ be the sequence of monic polynomials of lowest
degree that satisfy the following non-hermitian orthogonality conditions:
\begin{equation}\label{orthog:Qn}
\int_{\Gamma_{0}} Q_{n}(z)\,z^{l}\,d s_{j}(z)=0,\qquad l=0,\ldots,\left\lfloor
\frac{n-j-1}{p}\right\rfloor,\qquad 0\leq j\leq p-1,
\end{equation}
where the measures $s_{j}$ are defined in \eqref{def:sj}
\end{definition}
In more detail, \eqref{orthog:Qn} asserts that the polynomial $Q_{mp+r}$ must
satisfy the orthogonality relations
\[
\int_{\Gamma_{0}} Q_{mp+r}(z)\,z^{l}\,ds_{j}(z)=0,\qquad l=0,\ldots,m-1,\quad
0\leq j\leq p-1,
\]
\[
\int_{\Gamma_{0}} Q_{mp+r}(z)\,z^{m}\,ds_{j}(z)=0,\qquad 0\leq j\leq
r-1.
\]

The goal of \cite{LopMin} was to investigate algebraic properties and the weak
asymptotic behavior of the sequence of multi-orthogonal polynomials $(Q_{n})$.
In the following result we summarize some of the properties proved in \cite{LopMin}.

\begin{proposition}\label{prop:Qnsummary}
The following properties hold:
\begin{itemize}
\item[1)] For each $n\geq 0$, the polynomial $Q_{n}$ has maximal degree $n$.
\item[2)] If $n\equiv \ell \mod (p+1)$, $0\leq \ell\leq p$, then there exists
a monic polynomial $\mathcal{Q}_{d}$ of degree $d=\frac{n-\ell}{p+1}$ such that
\begin{equation}\label{eq:decompQn}
Q_{n}(z)=z^{\ell} \mathcal{Q}_{d}(z^{p+1}),
\end{equation}
where the zeros of $\mathcal{Q}_{d}$ are all simple and located in $(a_{0},b_{0})$. In particular,
the zeros of $Q_{n}$ are located in the star-like set $S_{+}$.
\item[3)] The polynomials $Q_{n}$ satisfy the following three-term recurrence relation of order $p+1$:
\begin{equation}\label{threetermrec}
z Q_{n}(z)=Q_{n+1}(z)+a_{n}\,Q_{n-p}(z),\qquad n\geq p,\qquad
a_{n}\in\mathbb{R},
\end{equation}
where
\begin{equation}\label{eq:firstQns}
Q_{\ell}(z)=z^{\ell},\qquad \ell=0,\ldots,p.
\end{equation}
\item[4)] The recurrence coefficients $a_{n}$ in \eqref{threetermrec} are all positive, i.e. $a_{n}>0$ for all $n\geq p$.
\item[5)] For every $n\geq p+1$, the non-zero roots of the polynomials $Q_{n}$ and $Q_{n+1}$ interlace on $\Gamma_{0}$.
\end{itemize}
\end{proposition}

To check the validity of the statements in Proposition~\ref{prop:Qnsummary}, we refer the reader
to the following results in \cite{LopMin}: Propositions 2.16 and 3.1, Theorem 3.5 and Corollary 3.6.

As it is well-known in the theory of multi-orthogonal polynomials, the so-called functions
of the second kind play a crucial role in the asymptotic analysis. These functions are defined next.

\begin{definition}\label{definitionPsi}
Set $\Psi_{n,0}=Q_{n}$ and let
\[
\Psi_{n,k}(z)=\int_{\Gamma_{k-1}}\frac{\Psi_{n,k-1}(t)}{z-t}\,d\sigma_{k-1}(t),
\qquad k=1,\ldots,p.
\]
\end{definition}

Observe that for each $k=1,\ldots,p$, $\Psi_{n,k}$ is analytic in $\overline{\mathbb{C}}\setminus
\Gamma_{k-1}$.

The functions $\Psi_{n,k}$ were also investigated in \cite{LopMin}, and some of
the properties found there are stated in the following result, see
Propositions 2.5--2.7 in \cite{LopMin}.

\begin{proposition}\label{prop:Psinksummary}
The following properties hold:
\begin{itemize}
\item[1)] For each $k=0,\ldots,p-1$, the function $\Psi_{n,k}$ satisfies the
orthogonality conditions
\begin{equation}\label{orthog:Psink}
\int_{\Gamma_{k}} \Psi_{n,k}(z)\,z^{l}\,ds_{k,j} (z)=0,\qquad 0\leq l\leq \lf
\frac{n-j-1}{p} \rf,\qquad k\leq j\leq p-1.
\end{equation}
\item[2)] Let $\omega:=e^{\frac{2\pi i}{p+1}}$. For each $k=0,\ldots,p,$ we have $\Psi_{n,k}(z)=\omega^{n-k} \Psi_{n,k}(z)$.
\item[3)] Assume that $n \equiv \ell \mod (p+1)$ with $0\leq \ell\leq p$. Then, for
each $k=1,\ldots, p$ we have
\begin{equation}\label{altintrep:Psink}
\Psi_{n,k}(z)=
z^{p-s}\,\int_{\Gamma_{k-1}}\frac{\Psi_{n,k-1}(t)\,t^{s}}{z^{p+1}-t^{p+1}}\,
d\sigma_{k-1}(t),
\end{equation}
where $s$ is the only integer in $\{0,\ldots,p\}$ such that $s\equiv
k-1-\ell\mod (p+1)$, that is,
\begin{equation}\label{param:s}
s=\begin{cases}
k-1-\ell, & \ell< k,\\
p+k-\ell, & k \leq \ell.
\end{cases}
\end{equation}
\end{itemize}
\end{proposition}

Some observations on Proposition~\ref{prop:Psinksummary} are appropriate at this point.
First, \eqref{orthog:Psink} shows that the function $\Psi_{n,k}$, $0\leq k\leq p-1$,
satisfies multiple orthogonality conditions similar to those satisfied by $Q_n$
but with respect to the Nikishin system given by the $k$th row of the
hierarchy \eqref{hiers}. Formula \eqref{altintrep:Psink} allows us to find a
representation of $\Psi_{n,k}$ that is similar to the representation of
$Q_{n}$ in \eqref{eq:decompQn}. The functions that are necessary
for this representation are defined next.

\begin{definition}\label{definitionpsi}
Set $\psi_{n,0}:=\mathcal{Q}_d$, where $\mathcal{Q}_d$ is the polynomial that appears in the right-hand side of \eqref{eq:decompQn}. For $1\leq k\leq p$, let
$\psi_{n,k}$ be the function analytic in $\mathbb{C}\setminus
[a_{k-1},b_{k-1}]$ defined by
\[
\psi_{n,k}(z)=\begin{cases}
z\int_{\Gamma_{k-1}}\frac{\Psi_{n,k-1}(t)\,t^{k-1-\ell}}{z-t^{p+1}}\,d\sigma_{
k-1}(t), & \ell< k,\\[1em]
\int_{\Gamma_{k-1}}\frac{\Psi_{n,k-1}(t)\,t^{p+k-\ell}}{z-t^{p+1}}\,d\sigma_{k-1
}(t), & k\leq \ell,
\end{cases}
\]
where, as before, $n\equiv \ell \mod (p+1)$, $0\leq \ell\leq p$.
\end{definition}

Indeed, the following important result is now an immediate consequence of
the above definition and \eqref{altintrep:Psink}--\eqref{param:s}.

\begin{corollary}\label{coroPsi} Suppose $n
\equiv \ell \mod (p+1)$ with $0\leq \ell\leq p$, and define
\begin{equation}\label{varymeas:sigma}
d\sigma_{n,k}(\tau):=\begin{cases}
d\sigma_{k}^{*}(\tau), & \ell\leq k,\\
\tau\,d\sigma_{k}^{*}(\tau), & k<\ell.
\end{cases}
\end{equation}
Then,
\begin{equation}\label{modifiedPsink}
z^{k-\ell}\,\Psi_{n,k}(z)=\psi_{n,k}(z^{p+1}), \quad 0\leq k\leq p,
\end{equation}
and for all $1\leq k\leq p$,
\begin{equation}
\psi_{n,k}(z)=\begin{cases}\label{recurrenceforpsipequena}
z\int_{a_{k-1}}^{b_{k-1}}\frac{\psi_{n,k-1}(\tau)}{z-\tau}\,d\sigma_{n,k-1}(\tau), & \ell<k,\\[1em]
\int_{a_{k-1}}^{b_{k-1}}\frac{\psi_{n,k-1}(\tau)}{z-\tau}\,d\sigma_{n,k-1}(\tau), & k\leq \ell.
\end{cases}
\end{equation}
\end{corollary}

We will also refer to the functions $\psi_{n,k}$ as functions of the second kind.

We have seen that the functions $\Psi_{n,k}$ satisfy orthogonality relations
with respect to the hierarchy \eqref{hiers}. An important property is that the associated
functions $\psi_{n,k}$ do the same with respect to the hierarchy
\eqref{hiermu}. We reproduce this property here, which is Proposition 2.10 in \cite{LopMin}.

\begin{proposition}
Let $0\leq k\leq p-1$ and assume that $n\equiv \ell \mod (p+1)$ with $0\leq
\ell\leq p$. Then the function $\psi_{n,k}$ satisfies the following
orthogonality conditions:
 \begin{equation}\label{orthogredPsink}
\int_{a_{k}}^{b_{k}}\psi_{n,k}(\tau)\,\tau^{s}\,d\mu_{k,j}(\tau)=0,\quad
\left\lceil\frac{\ell-j}{p+1}\right\rceil\leq s\leq \left\lfloor
\frac{n+p\ell-1-j(p+1)}{p(p+1)}\right\rfloor,\quad k\leq j\leq p-1.
\end{equation}
\end{proposition}

\subsection{Counting the number of orthogonality conditions}\label{counting}

For the asymptotic analysis of the multi-orthogonal polynomials and the
functions of the second kind, it is crucial to have a control on the
total number of orthogonality conditions in \eqref{orthogredPsink}. We define this quantity next
in the same way it was defined in \cite{LopMin}.

\begin{definition}
Let  $n$ be a nonnegative integer, and let $\ell$ be the integer satisfying $n\equiv \ell\mod (p+1)$, $0\leq\ell\leq p$. For each
$0\leq j\leq p-1$, let $M_j=M_j(n)$ be the number of integers $s$
satisfying the inequalities
\begin{equation}\label{countingcond}
\left\lceil\frac{\ell-j}{p+1}\right\rceil\leq s\leq \left\lfloor
\frac{n+p\ell-1-j(p+1)}{p(p+1)}\right\rfloor.
\end{equation}
For each $0\leq k\leq p-1$, we define
\begin{equation}\label{def:Znk}
 Z(n,k):=\sum_{j=k}^{p-1}M_j.
\end{equation}
Also, we convene to set $Z(n,p):=0$.
\end{definition}

It is clear from the definition that for every $n$,
\begin{equation}\label{eq:ineqZnk}
Z(n,k)\geq Z(n,k+1), \quad 0\leq k\leq p-2,
\end{equation}
and
\[
Z(n,k)-Z(n,k+1)=\#\left\{s:\left\lceil\frac{\ell-k}{p+1}\right\rceil\leq s\leq
\left\lfloor \frac{n+p\ell-1-k(p+1)}{p(p+1)}\right\rfloor\right\}.
\]

Exact formulas for the quantity $Z(n,k)$ are involved; the reader may look at the expressions that appear
in Lemma 2.13 and Proposition 2.17 from \cite{LopMin}. However, these formulas are not needed in the present paper.
We just note here that
\begin{equation}\label{asympZnk}
Z(n,k)=\frac{n(p-k)}{p(p+1)}+O(1),\qquad n\rightarrow\infty.
\end{equation}
Later it will be important for our study of ratio asymptotics to analyze the difference $Z(n+1,k)-Z(n,k)$.

\subsection{More properties of the functions of the second kind, and the polynomials $P_{n,k}$}

In this paper we will need some more properties of the functions $\psi_{n,k}$ that were proved in \cite{LopMin}.
We gather some of them in the next result.

\begin{proposition}\label{prop:lpsinksummary}
The following properties hold:
\begin{itemize}
\item[1)] Let $1\leq k\leq p$, and suppose that $n\equiv\ell\mod(p+1)$. Then, as
$z\to\infty$,
\begin{align}\label{decayinfpsink}
\psi_{n,k}(z) & =O(z^{-N(n,k)}),
\end{align}
where
\begin{equation}\label{rel:NnkZnk}
 N(n,k)=\begin{cases}
         Z(n,k-1)-Z(n,k), & \ell<k,\\
         Z(n,k-1)-Z(n,k)+1, & k\leq \ell,
        \end{cases}
\end{equation}
recall $Z(n,p)=0$.
\item[2)] For each $n\geq 0$ and $k=0,\ldots,p-1$, the function $\psi_{n,k}$ has exactly $Z(n,k)$
zeros in $\mathbb{C}\setminus([a_{k-1},b_{k-1}]\cup\{0\})$; they are all simple
and lie in the open interval $(a_{k},b_{k})$. The function $\psi_{n,p}$ has no
zeros in $\mathbb{C}\setminus([a_{p-1},b_{p-1}]\cup\{0\})$.
\item[3)] Let $a_n$, $n\geq p$, be the coefficients of the recurrence
relation \eqref{threetermrec}. For every $n\geq p$, $0\leq k\leq p$, we have
\begin{equation}\label{threetermrecsecondkind}
z\Psi_{n,k}(z)=\Psi_{n+1,k}(z)+a_n\Psi_{n-p,k}(z),
\end{equation}
and if $n\equiv\ell\mod(p+1)$, $0\leq \ell \leq p-1$, then
\begin{equation}\label{threetermrecsecondkindpequena1}
\psi_{n,k}(z)=\psi_{n+1,k}(z)+a_n\psi_{n-p,k}(z),
\end{equation}
while if $n\equiv p\mod(p+1)$, then
\begin{equation}\label{threetermrecsecondkindpequena2}
z\psi_{n,k}(z)=\psi_{n+1,k}(z)+a_n\psi_{n-p,k}(z).
\end{equation}
\end{itemize}
\end{proposition}

See Propositions 2.18, 2.19 and 3.2 in \cite{LopMin} for a proof of these properties.

As in \cite{LopMin}, an important role in the asymptotic analysis will be played
by certain monic polynomials associated with the functions $\psi_{n,k}$ that we define next.

\begin{definition}\label{Def:polyPnk}
For any integers $n\geq 0$ and $k$ with $0\leq k\leq p-1$, let $P_{n,k}$ denote
the monic polynomial whose zeros are the zeros of $\psi_{n,k}$ in
$(a_{k},b_{k})$. For convenience we also define the polynomials $P_{n,-1}\equiv
1$, $P_{n,p}\equiv 1$.
\end{definition}

Hence by Proposition~\ref{prop:lpsinksummary} we know that $P_{n,k}$ has
degree $Z(n,k)$ and all its zeros are simple. Recall that by Definition~\ref{definitionpsi}, $P_{n,0}=\psi_{n,0}$
is the polynomial $\mathcal{Q}_{d}$ that appears in \eqref{eq:decompQn}, and therefore
\begin{equation}\label{eq:formZn0}
Z(n,0)=\deg(P_{n,0})=\left\lfloor\frac{n}{p+1}\right\rfloor.
\end{equation}

The main purpose for introducing the polynomials $P_{n,k}$ is to prove certain
orthogonality conditions satisfied by the functions $\psi_{n,k}$ with respect
to varying measures involving these polynomials.

\begin{proposition}
Let $0\leq k\leq p-1$. Then, the function $\psi_{n,k}$ satisfies the following orthogonality conditions:
\begin{equation}\label{varyorthog:psink:1}
\int_{a_{k}}^{b_{k}}\psi_{n,k}(\tau)\,\tau^{s}\,\frac{d\sigma_{n,k}(\tau)}{P_{n,
k+1}(\tau)}=0,\qquad s=0,\ldots,Z(n,k)-1,
\end{equation}
recall \eqref{varymeas:sigma}.
\end{proposition}

For a justification of \eqref{varyorthog:psink:1}, see Proposition 2.21 in \cite{LopMin}.

Observe that by definition of $Z(n,k)$, the total number of
orthogonality conditions in \eqref{orthogredPsink} agrees with
the number of orthogonality conditions in \eqref{varyorthog:psink:1},
but the advantage of \eqref{varyorthog:psink:1} is clear since it
involves only one orthogonality measure.

\subsection{The auxiliary functions $H_{n,k}$}

In this subsection we introduce certain functions that will play an important role
in the analysis that will follow.

\begin{definition}\label{defHnk}For integers $n\geq 0$ and $0\leq k\leq p$,
set
\begin{equation}\label{eq:def:Hnk}
H_{n,k}(z):=\frac{P_{n,k-1}(z)\,\psi_{n,k}(z)}{P_{n,k}(z)}.
\end{equation}
\end{definition}

Note that $H_{n,0}\equiv 1$. Since the zeros of $P_{n,k}$ are zeros of
$\psi_{n,k}$ outside $[a_{k-1},b_{k-1}]$, we have
\[
H_{n,k}\in\mathcal{H}(\mathbb{C}\setminus[a_{k-1},b_{k-1}]),\quad 1\leq k\leq
p.
\]

Putting together \eqref{varymeas:sigma}, \eqref{varyorthog:psink:1}, and \eqref{eq:def:Hnk}, we readily obtain
the following result.

\begin{proposition}\label{orthoPnk}
For any $k=0,\ldots,p-1$, the polynomial $P_{n,k}$ satisfies the following
orthogonality conditions:
\[
\int_{a_{k}}^{b_{k}}P_{n,k}(\tau)\,\tau^{s}\,\frac{H_{n,k}(\tau)\,d\sigma_{n,k}
(\tau)}{P_{n,k-1}(\tau)\,P_{n,k+1}(\tau)}=0,\qquad s=0,\ldots,Z(n,k)-1.
\]
Recall that $P_{n,-1}, P_{n,p}\equiv 1$.
\end{proposition}

The function $H_{n,k}$ has an integral representation that is analogous
to the integral representation of the function $\psi_{n,k}$ in \eqref{recurrenceforpsipequena}.

\begin{proposition}
Let $1\leq k\leq p$ and $n\equiv \ell \mod(p+1)$, $0\leq \ell\leq p$. Then,
\begin{equation}\label{intrep:Hnk}
H_{n,k}(z)=\begin{cases}
z\int_{a_{k-1}}^{b_{k-1}}\frac{P_{n,k-1}^{2}(\tau)}{z-\tau}\,\frac{H_{n,k-1}
(\tau)\,d\sigma_{n,k-1}(\tau)}{P_{n,k-2}(\tau)\,P_{n,k}(\tau)}, & \ell<k,\\[1em]
\int_{a_{k-1}}^{b_{k-1}}\frac{P_{n,k-1}^{2}(\tau)}{z-\tau}\,\frac{H_{n,k-1}
(\tau)\,d\sigma_{n,k-1}(\tau)}{P_{n,k-2}(\tau)\,P_{n,k}(\tau)}, &
k\leq \ell.\\
\end{cases}
\end{equation}
\end{proposition}

Formula \eqref{intrep:Hnk} was proved in \cite[Proposition 2.25]{LopMin}.

\subsection{Normalization}

In this subsection we introduce a convenient normalization of the polynomials
$P_{n,k}$ and the functions $H_{n,k}$.

It follows from the definition of the functions $H_{n,k}$ and the polynomials
$P_{n,k}$ that the measures
\[
\frac{H_{n,k}(\tau)\,d\sigma_{n,k}(\tau)}{P_{n,k-1}(\tau)\,P_{n,k+1}(\tau)},\qquad 0\leq k\leq p-1,
\]
have constant sign on the interval $[a_{k}, b_{k}]$. We then denote by
\[
\frac{|H_{n,k}(\tau)|\,d|\sigma_{n,k}|(\tau)}{|P_{n,k-1}(\tau)\,P_{n,k+1}(\tau)|
}
\]
the positive normalization of this measure and we have
\begin{equation}\label{eq:normvaryorthog:Pnk}
\int_{a_{k}}^{b_{k}}P_{n,k}(\tau)\,\tau^{s}\,\frac{|H_{n,k}(\tau)|\,d|\sigma_{n,
k}|(\tau)}{|P_{n,k-1}(\tau)\,P_{n,k+1}(\tau)|}=0,\quad
s=0,\ldots,Z(n,k)-1,\quad k=0,\ldots,p-1.
\end{equation}

Let
\begin{align}
K_{n,-1} & :=1,\quad K_{n,p}:=1,\label{eq:def:Knm1}\\
K_{n,k} & :=
\left(\int_{a_{k}}^{b_{k}}P_{n,k}^{2}(\tau)\,\frac{|H_{n,k}(\tau)|\,d|\sigma_{n,
k}|(\tau)}{|P_{n,k-1}(\tau)\,P_{n,k+1}(\tau)|}\right)^{-1/2},\qquad
k=0,\ldots,p-1,\label{eq:def:Knk}
\end{align}
and we also define the constants
\begin{equation}\label{def:kappank}
\kappa_{n,k}:=\frac{K_{n,k}}{K_{n,k-1}},\qquad k=0,\ldots,p.
\end{equation}

\begin{definition}
For $k=0,\ldots,p,$ we define
\begin{align}
p_{n,k} & :=\kappa_{n,k}\,P_{n,k},\label{eq:def:lpnk}\\
h_{n,k} & := K_{n,k-1}^{2}\,H_{n,k},\label{eq:def:hnk}
\end{align}
where the constants $\kappa_{n,k}$ and $K_{n,k}$ are given in
\eqref{def:kappank} and \eqref{eq:def:Knm1}--\eqref{eq:def:Knk}, respectively.
\end{definition}

We will denote by $\nu_{n,k}$ the measure on $[a_{k},b_{k}]$ given by
\begin{equation}\label{def:measnunk}
d\nu_{n,k}(\tau):=\frac{h_{n,k}(\tau)\,d\sigma_{n,k}(\tau)}{P_{n,k-1}(\tau)\,P_{
n,k+1}(\tau)},\qquad k=0,\ldots,p-1.
\end{equation}
Again this measure has constant sign in $[a_{k},b_{k}]$, and we will denote by
$\varepsilon_{n,k}$ its sign and by $|\nu_{n,k}|$ its positive normalization,
hence
\begin{equation}\label{positivenunk}
d|\nu_{n,k}|(\tau)=\frac{|h_{n,k}(\tau)|\,d|\sigma_{n,k}|(\tau)}{|P_{n,k-1}
(\tau)\,P_{n,k+1}(\tau)|}=\varepsilon_{n,k}\,\frac{h_{n,k}(\tau)\,d\sigma_{n,k}
(\tau)}{P_{n,k-1}(\tau)\,P_{n,k+1}(\tau)}.
\end{equation}
An exact formula for $\varepsilon_{n,k}$ is given in \eqref{eq:formulaenk}.

\begin{proposition}
For each $k=0,\ldots,p-1,$ the polynomial $p_{n,k}$ defined in
\eqref{eq:def:lpnk} satisfies the following:
\begin{align}
\int_{a_{k}}^{b_{k}} p_{n,k}(\tau)\,\tau^{s}\,d|\nu_{n,k}|(\tau) & =0,\qquad
s=0,\ldots,Z(n,k)-1,\label{orthog:lpnk}\\
\int_{a_{k}}^{b_{k}} p_{n,k}^{2}(\tau)\,d|\nu_{n,k}|(\tau) &
=1,\label{orthonorm:lpnk}
\end{align}
that is, $p_{n,k}$ is the orthonormal polynomial of degree $Z(n,k)$ with
respect to the positive measure $|\nu_{n,k}|$.

For each $k=1,\ldots,p$, the function $h_{n,k}$ defined in \eqref{eq:def:hnk}
satisfies
\begin{equation}\label{intrep:hnk}
h_{n,k}(z)=\begin{cases}
\varepsilon_{n,k-1}\,z\int_{a_{k-1}}^{b_{k-1}}\frac{p_{n,k-1}^{2}(\tau)}{z-\tau}
\,d|\nu_{n,k-1}|(\tau) & \ell<k,\\[1em]
\varepsilon_{n,k-1}\int_{a_{k-1}}^{b_{k-1}}\frac{p_{n,k-1}^{2}(\tau)}{z-\tau}\,
d|\nu_{n,k-1}|(\tau) & k\leq \ell.
\end{cases}
\end{equation}
\end{proposition}
\begin{proof}
The orthogonality conditions \eqref{orthog:lpnk} are obvious in view of
\eqref{eq:normvaryorthog:Pnk}. The formulas \eqref{orthonorm:lpnk} and
\eqref{intrep:hnk} follow immediately from
\eqref{eq:def:lpnk}--\eqref{eq:def:hnk}, \eqref{def:kappank},
\eqref{eq:def:Knm1}--\eqref{eq:def:Knk} and \eqref{intrep:Hnk}.
\end{proof}

\section{Main results}\label{section:Main}

In this paper we initiate our analysis with the study of the difference $Z(n+1,k)-Z(n,k)=\deg(P_{n+1,k})-\deg(P_{n,k})$, see Section~\ref{sec:diffZnk}. Recall that the zeros of the polynomials $P_{n,k}$ are all simple and lie in the interval $(a_{k},b_{k})$, cf. Definition~\ref{Def:polyPnk} and Proposition~\ref{prop:lpsinksummary}. We show that for each fixed $k=0,\ldots,p-1,$ the difference $Z(n+1,k)-Z(n,k)$ is periodic in $n$ with period $p(p+1)$ and takes only the values $\{-1,0,1\}$, see Lemma~\ref{lem:valuesdiffZnk}.

Our first main result is the interlacing property of the zeros of $P_{n,k}$. This property is proved in Section~\ref{sec:interlacing} and is a consequence of an auxiliary result on the zeros of the function $G_{n,k}$ defined in \eqref{def:Gnk}, see Corollary~\ref{cor:interl}.

\begin{theorem}\label{theo:interl}
Let $(s_{0},\ldots,s_{p-1})=\mathcal{N}(\sigma_{0},\ldots,\sigma_{p-1})$ be a Nikishin system on $\Gamma_{0}$, defined as indicated in Section~\ref{sect:defNik}, and let $(P_{n,k})$ be the system of associated polynomials introduced in Definition~\ref{Def:polyPnk}.
Then for each $k=0,\ldots,p-1$, the zeros of $P_{n,k}$ and $P_{n+1,k}$ interlace; that is, between two consecutive zeros of $P_{n,k}$
there is exactly one zero of $P_{n+1,k}$ and vice versa.
\end{theorem}

The following theorem is our main asymptotic result, from which we derive all other asymptotic formulas. In all these formulas, convergence is uniform on compact subsets of the indicated regions, and the periodicity modulo $p(p+1)$ is preserved. Theorem~\ref{ratioP} and the corollaries that follow are all proved in Section~\ref{sec:ratioasymp}. 

\begin{theorem}\label{ratioP}
Assume that for each $k=0,\ldots,p-1$, the measure $\sigma_{k}^{*}$ defined in \eqref{def:sigma:star} has positive Radon-Nikodym derivative with respect to Lebesgue measure a.e. on $[a_{k},b_{k}]$. Then, for each fixed $0\leq \rho\leq p(p+1)-1$,
\begin{equation}\label{limitP}
\lim_{\lambda \to \infty}\frac{P_{\lambda p(p+1)+\rho+1,k}(z)}{P_{\lambda p(p+1)+\rho,k}(z)}=\widetilde{F}_{k}^{(\rho)}(z),\qquad z\in\mathbb{C}\setminus[a_{k},b_{k}], \qquad k=0,\ldots,p-1,
\end{equation}
and
\begin{equation}\label{limitP2}
\lim_{n \to \infty}\frac{P_{n + p(p+1) ,k}(z)}{P_{n,k}(z)}=\prod_{\rho =0}^{p(p+1)-1}\widetilde{F}_{k}^{(\rho)}(z),\qquad z\in\mathbb{C}\setminus[a_{k},b_{k}], \qquad k=0,\ldots,p-1,
\end{equation}
where $\widetilde{F}_k^{(\rho)}$ is $F_k^{(\rho)}$ divided by its leading coefficient (in the Laurent series expansion at infinity) and $(F_k^{(\rho)})_{k=0}^{p-1}$ is a collection of analytic functions that satisfies, for the given value of $\rho$, the properties stated in Lemma \ref{boundary}.
\end{theorem}

We describe next the ratio asymptotic behavior of the multiple orthogonal polynomials $(Q_{n})_{n=0}^{\infty}$ and the asymptotic behavior of the associated recurrence coefficients $(a_{n})_{n=p}^{\infty}$.

\begin{corollary}\label{cor:ratioQnan}
Assume that for each $k=0,\ldots,p-1$, the measure $\sigma_{k}^{*}$ has positive Radon-Nikodym derivative with respect to Lebesgue measure a.e. on $[a_{k},b_{k}]$. Then the sequence $(Q_{n})$ of multi-orthogonal polynomials introduced in Definition~\ref{def:MOP} and the sequence $(a_{n})$ of recurrence coefficients in \eqref{threetermrec} satisfy the following:

Let $0\leq \rho\leq p(p+1)-1$ be fixed. If $\rho\not\equiv p \mod (p+1)$, then
\begin{equation}\label{eq:ratioQn:1}
\lim_{\lambda\rightarrow\infty}\frac{Q_{\lambda p(p+1)+\rho+1}(z)}{Q_{\lambda p(p+1)+\rho}(z)}=z\,\widetilde{F}_{0}^{(\rho)}(z^{p+1}),\qquad z\in\mathbb{C}\setminus\Gamma_{0},
\end{equation}
and if $\rho\equiv p\mod (p+1)$, then
\begin{equation}\label{eq:ratioQn:2}
\lim_{\lambda\rightarrow\infty}\frac{Q_{\lambda p(p+1)+\rho+1}(z)}{Q_{\lambda p(p+1)+\rho}(z)}=\frac{\widetilde{F}_{0}^{(\rho)}(z^{p+1})}{z^{p}},\qquad z\in\mathbb{C}\setminus(\Gamma_{0}\cup\{0\}).
\end{equation}
We have
\begin{equation}\label{eq:limreccoeff}
\lim_{\lambda\rightarrow\infty}a_{\lambda p(p+1)+\rho}=a^{(\rho)},
\end{equation}
where the limiting values $a^{(\rho)}$ appear in the Laurent expansion at infinity of $\widetilde{F}_{0}^{(\rho)}$ as follows:
\begin{equation}\label{eq:LaurentFexp}
\widetilde{F}_{0}^{(\rho)}(z)=\begin{cases}
1-a^{(\rho)} z^{-1}+O\left(z^{-2}\right), & \ \mbox{if} \ \rho\not\equiv p \mod (p+1),\\
z-a^{(\rho)}+O\left(z^{-1}\right), & \ \mbox{if} \ \rho\equiv p \mod (p+1).
\end{cases}
\end{equation}
\end{corollary}

In the next result we describe the ratio asymptotic behavior of the functions of the second kind $(\psi_{n,k})$, $(\Psi_{n,k})$ and the normalized polynomials $(p_{n,k})$.

\begin{corollary}\label{ratioP3}
Assume that for each $k=0,\ldots,p-1$, the measure $\sigma_{k}^{*}$ has positive Radon-Nikodym derivative with respect to Lebesgue measure a.e. on $[a_{k},b_{k}]$. Then, for each fixed $0\leq \rho \leq p(p+1)-1$,
\begin{equation}\label{limitP3}
\lim_{\lambda \to \infty}\frac{\kappa_{\lambda p(p+1)+\rho+1,k}}{\kappa_{\lambda p(p+1)+\rho,k}}= \kappa_k^{(\rho)}, \qquad k=0,\ldots,p-1,
\end{equation}
\begin{equation}\label{limitP3+}
\lim_{\lambda \to \infty}\frac{K_{\lambda p(p+1)+\rho+1,k}}{K_{\lambda p(p+1)+\rho,k}}= \kappa_0^{(\rho)}\cdots\kappa_k^{(\rho)}, \qquad k=0,\ldots,p-1,
\end{equation}
and
\begin{equation}\label{limitP4}
\lim_{\lambda \to \infty}\frac{p_{\lambda p(p+1)+\rho+1,k}(z)}{p_{\lambda p(p+1)+\rho,k}(z)}=\kappa_k^{(\rho)}\widetilde{F}_{k}^{(\rho)}(z),\qquad z\in\mathbb{C}\setminus[a_{k},b_{k}], \qquad k=0,\ldots,p-1,
\end{equation}
where
\begin{equation}
\label{kappa}
\kappa_k^{(\rho)} = \frac{c_k^{(\rho)}}{(c_{k-1}^{(\rho)}c_{k+1}^{(\rho)})^{1/2}},
\end{equation}
and $(c_k^{(\rho)}), k=0,\ldots,p-1$ is the unique solution of the system of equations \eqref{eq:cs} $(c_{-1}^{(\rho)} = c_{p}^{(\rho)} =1 )$.

Regarding the functions of the second kind $(\psi_{n,k})$ and $(\Psi_{n,k})$, for $k=1,\ldots,p$ we have
\begin{equation}
\label{second}
\lim_{\lambda \to \infty } \frac{\psi_{\lambda p(p+1) + \rho +1,k}(z)}{\psi_{\lambda p(p+1) + \rho,k}(z)} = \frac{\varepsilon_{k}^{(\rho)}\,h_{k}^{(\rho)}(z)}{(\kappa_0^{(\rho)}\cdots\kappa_{k-1}^{(\rho)})^2} \frac{\widetilde{F}_k^{(\rho)}(z)}{\widetilde{F}_{k-1}^{(\rho)}(z)}, \qquad z\in\mathbb{C}\setminus([a_{k-1},b_{k-1}]\cup[a_{k},b_{k}]\cup\{0\}),
\end{equation}
where $\varepsilon_{k}^{(\rho)}$ is either $1$ or $-1$ depending on $k$ and $\rho$, and
\begin{equation}\label{def:hkrho}
h_{k}^{(\rho)}(z)=\begin{cases}
z, & \mbox{if}\ \rho\equiv p \mod (p+1),\\
z^{-1}, & \mbox{if}\ \rho\equiv k-1 \mod (p+1),\\
1, & \mbox{otherwise}.
\end{cases}
\end{equation}
In \eqref{second} we use the convention $\widetilde{F}_{p}^{(\rho)}\equiv 1$. Under the same assumptions on $\rho$ and $k$ we have
\begin{equation}\label{asymp:Psink}
\lim_{\lambda\rightarrow\infty}\frac{\Psi_{\lambda p(p+1)+\rho+1,k}(z)}{\Psi_{\lambda p(p+1)+\rho,k}(z)}=\frac{\varepsilon_{k}^{(\rho)}\,g_{k}^{(\rho)}(z)}{(\kappa_0^{(\rho)}\cdots\kappa_{k-1}^{(\rho)})^2} \frac{\widetilde{F}_k^{(\rho)}(z^{p+1})}{\widetilde{F}_{k-1}^{(\rho)}(z^{p+1})}, \qquad z\in\mathbb{C}\setminus(\Gamma_{k-1}\cup\Gamma_{k}\cup\{0\}),
\end{equation}
where
\[
g_{k}^{(\rho)}(z)=\begin{cases}
z^{-p}, & \mbox{if} \ \rho\equiv k-1 \mod (p+1),\\
z, & \mbox{otherwise}.
\end{cases}
\]
\end{corollary}

For an exact formula of $\varepsilon_{k}^{(\rho)}$, see Remark~\ref{rmksignepsilon}.

\section{Analysis of $Z(n+1,k)-Z(n,k)$ and some consequences}\label{sec:diffZnk}

The study of ratio asymptotics requires as a preliminary step the analysis of $Z(n+1,k)-Z(n,k)$, which is the difference in the degrees of the consecutive polynomials $P_{n+1,k}$ and $P_{n,k}$. Since we are using in this section both indices $n$ and $n+1$, we will indicate below the dependence with respect to $n$ of the quantities that appear in Lemma~\ref{lem:valuesdiffZnk}. Thus we will write for instance $\ell(n)$ for the integer satisfying
\[
n \equiv  \ell(n) \mod (p+1),\qquad 0\leq \ell(n)\leq p,
\]
and so on.

Throughout this section we decompose $n$ in the form
\begin{equation}\label{decompmodpp}
n=\lambda p(p+1)+\rho,
\end{equation}
where $\lambda=\lambda(n)$, $\rho=\rho(n)$ are integers and $0\leq\rho\leq p(p+1)-1$.
We also decompose $\rho$ modulo $p+1$ and write
\begin{equation}\label{eq:decomp:rho}
\rho(n)=\eta(p+1)+\ell,\qquad \eta=\eta(n)\in\{0,\ldots,p-1\},\qquad \ell=\ell(n)\in\{0,\ldots,p\}.
\end{equation}
Given two integers $n$ and $n'$ with $n\leq n'$, in this section and in the rest of the paper we will use for convenience the notation $[n:n']$ to indicate the set $\{s\in\mathbb{Z}: n\leq s\leq n'\}$.

Before describing the difference $Z(n+1,k)-Z(n,k)$, we make some relevant observations. The lower bound in \eqref{countingcond} is simply
\[
\left\lceil\frac{\ell(n)-j}{p+1}\right\rceil=\begin{cases}
0, & \mathrm{if}\ \ell(n)\leq j,\\
1, & \mathrm{if}\ j<\ell(n).
\end{cases}
\]
Concerning the upper bound in \eqref{countingcond}, using \eqref{decompmodpp} and \eqref{eq:decomp:rho} we get
\begin{equation}\label{upbd1}
\left\lfloor\frac{n+p\,\ell(n)-1-j(p+1)}{p(p+1)}\right\rfloor=\left\lfloor \lambda+\frac{\eta(n)+\ell(n)-j}{p}-\frac{1}{p(p+1)}\right\rfloor,
\end{equation}
where $\lambda$ corresponds to $n$, but we prefer not to write $\lambda(n)$ in \eqref{upbd1} and below. If $\ell(n)<p$, then $\ell(n+1)=\ell(n)+1$,
and in this case we have that the upper bound in \eqref{countingcond} that corresponds to $n+1$ is
\begin{equation}\label{upbd2}
\left\lfloor\frac{n+1+p\,\ell(n+1)-1-j(p+1)}{p(p+1)}\right\rfloor=\left\lfloor \lambda+\frac{\eta(n)+\ell(n)-j+1}{p}-\frac{1}{p(p+1)}\right\rfloor.
\end{equation}
However, if $\ell(n)=p$, then $\ell(n+1)=0$, and then
\begin{equation}\label{upbd3}
\left\lfloor\frac{n+1+p\,\ell(n+1)-1-j(p+1)}{p(p+1)}\right\rfloor=\left\lfloor \lambda+\frac{\eta(n)-j+1}{p}-\frac{1}{p(p+1)}\right\rfloor.
\end{equation}

Assume now that $\ell(n)<p$. Then we easily see that the quantities in \eqref{upbd1} and \eqref{upbd2}
are equal for all values of $j$ (recall $j\in[k:p-1]$) except those satisfying $\eta(n)+\ell(n)-j\in p\mathbb{Z}$.
In virtue of the restrictions
on the quantities $\eta$, $\ell$ and $j$,
the only possible exceptional cases are $j=\eta(n)+\ell(n)$ and $j=\eta(n)+\ell(n)-p$. In any of these
two exceptional cases, we have that the quantity in \eqref{upbd2} is one unit greater than the quantity in \eqref{upbd1}.
Also, as $j$ runs from $k$ to $p-1$, clearly $j$ can take at most one of these exceptional values. These are the key observations to keep in mind in order to prove the following result.

\begin{lemma}\label{lem:valuesdiffZnk}
Let $k\in[0:p-1]$ be fixed. Then, as a function of $n$, the expression $Z(n+1,k)-Z(n,k)$ is periodic with period $p(p+1)$,
and $Z(n+1,k)-Z(n,k)\in\{-1,0,1\}$ for all $n$. Moreover, using the notation \eqref{eq:decomp:rho} we have the following:

If $\ell(n)\in[0:p-1]$, then
\begin{equation}\label{eq:values:diffZnk:1}
Z(n+1,k)-Z(n,k)=\begin{cases}
0, & \mathrm{if}\ \eta(n)+\ell(n)\notin[k:p-1],\quad \ell(n)\in[0:k-1],\\[0.5em]
1, & \mathrm{if}\ \eta(n)+\ell(n)\in[k:p-1],\quad \ell(n)\in [0:k-1],\\[0.5em]
0,  & \mathrm{if}\ \eta(n)+\ell(n)\notin [p:p+k-1],\quad \ell(n)\in[k:p-1],\\[0.5em]
-1, &  \mathrm{if}\ \eta(n)+\ell(n)\in [p:p+k-1],\quad \ell(n)\in[k:p-1].
\end{cases}
\end{equation}
If $\ell(n)=p$, then
\begin{equation}\label{eq:values:diffZnk:2}
Z(n+1,k)-Z(n,k)=\begin{cases}
0, & \mathrm{if}\ \eta(n)\in[0:k-1],\\[0.5em]
1, & \mathrm{if}\ \eta(n)\in[k:p-1].
\end{cases}
\end{equation}
\end{lemma}
\begin{proof}
For convenience let us introduce the notation
\[
\varsigma(n,j):=\left\lceil\frac{\ell(n)-j}{p+1}\right\rceil,\qquad \kappa(n,j):=\left\lfloor\frac{n+p\,\ell(n)-1-j(p+1)}{p(p+1)}\right\rfloor.
\]

Assume that $\ell(n)<k$ and $\eta(n)+\ell(n)\in[k:p-1]$. Then we have that
\begin{equation}\label{relbounds:1}
\varsigma(n,j)=\varsigma(n+1,j)=0,\qquad \mbox{for all}\quad j\in[k:p-1],
\end{equation}
and
\begin{equation}\label{relbounds:2}
\kappa(n,j)=\kappa(n+1,j),
\qquad \mbox{for all}\quad j\in[k:p-1]\setminus\{\eta(n)+\ell(n)\}.
\end{equation}
As it was already observed, if $j=\eta(n)+\ell(n)$ then $\kappa(n+1,j)=\kappa(n,j)+1$. Therefore, by definition of $Z(n,k)$ and $Z(n+1,k)$, we obtain that if $\ell(n)<k$ and $\eta(n)+\ell(n)\in[k:p-1]$ then $Z(n+1,k)=Z(n,k)+1$.

Assume now that $\ell(n)<k$ and $\eta(n)+\ell(n)\notin[k:p-1]$. We still have the relation \eqref{relbounds:1}, and the first assumption shows that $\eta(n)+\ell(n)-p<\ell(n)<k$, hence $\eta(n)+\ell(n)-p\notin[k:p-1]$. Therefore, as $j$ runs from $k$ to $p-1$, $j$ does not take the exceptional values $\eta(n)+\ell(n)$ and $\eta(n)+\ell(n)-p$, which implies that $\kappa(n,j)=\kappa(n+1,j)$ for all $j\in[k:p-1]$. It follows that in this case $Z(n,k)=Z(n+1,k)$.

Now we assume that $k\leq \ell(n)\leq p-1$ and $\eta(n)+\ell(n)\in[k:p-1]$. Then
\begin{equation}\label{relbounds:3}
\varsigma(n,j)=\begin{cases}
1, & \ j\in[k:\ell(n)-1],\\[0.5em]
0, & \ j\in[\ell(n):p-1],
\end{cases}
\end{equation}
and
\begin{equation}\label{relbounds:4}
\varsigma(n+1,j)=\left\lceil\frac{\ell(n)+1-j}{p+1}\right\rceil=\begin{cases}
1, & \ j\in[k:\ell(n)],\\[0.5em]
0, & \ j\in[\ell(n)+1:p-1].
\end{cases}
\end{equation}
On the other hand, since $\eta(n)+\ell(n)\in[k:p-1]$, we see that \eqref{relbounds:2} holds and if $j=\eta(n)+\ell(n)$,
then $\kappa(n+1,j)=\kappa(n,j)+1$.

Assume additionally for the moment that $\eta(n)=0$. Then according to \eqref{relbounds:3}--\eqref{relbounds:4} and the above observations, we obtain that the intervals \eqref{countingcond} that correspond to $n$ and $n+1$ take the following form:
\begin{equation}\label{relbounds:5}
\begin{gathered}
\left[\varsigma(n,j):\kappa(n,j)\right]=\left[\varsigma(n+1,j):\kappa(n+1,j)\right]=[1:\lambda],\qquad j\in[k:\ell(n)-1],\\
[\varsigma(n,j),\kappa(n,j)]=[0:\lambda-1],\qquad [\varsigma(n+1,j),\kappa(n+1,j)]=[1:\lambda],\qquad \mbox{for}\ j=\ell(n),\\
\left[\varsigma(n,j):\kappa(n,j)\right]=\left[\varsigma(n+1,j):\kappa(n+1,j)\right]=[0:\lambda-1],\qquad j\in[\ell(n)+1:p-1].
\end{gathered}
\end{equation}
Therefore in this case we have $Z(n,k)=Z(n+1,k)$. Assume now that $\eta(n)>0$, while still assuming that $k\leq \ell(n)\leq p-1$ and $\eta(n)+\ell(n)\in[k:p-1]$. In this case the intervals \eqref{countingcond} for $n$ and $n+1$ have the form
\begin{equation}\label{relbounds:6}
\begin{gathered}
\left[\varsigma(n,j):\kappa(n,j)\right]=\left[\varsigma(n+1,j):\kappa(n+1,j)\right]=[1:\lambda],\qquad j\in[k:\ell(n)-1],\\
[\varsigma(n,j),\kappa(n,j)]=[0:\lambda],\qquad [\varsigma(n+1,j),\kappa(n+1,j)]=[1:\lambda],\qquad \mbox{for}\ j=\ell(n),\\
\left[\varsigma(n,j):\kappa(n,j)\right]=\left[\varsigma(n+1,j):\kappa(n+1,j)\right]=[0:\lambda],\qquad j\in[\ell(n)+1:\eta(n)+\ell(n)-1],\\
[\varsigma(n,j),\kappa(n,j)]=[0:\lambda-1],\qquad [\varsigma(n+1,j),\kappa(n+1,j)]=[0:\lambda],\qquad \mbox{for}\ j=\eta(n)+\ell(n),\\
\left[\varsigma(n,j):\kappa(n,j)\right]=\left[\varsigma(n+1,j):\kappa(n+1,j)\right]=[0:\lambda-1],\qquad j\in[\eta(n)+\ell(n)+1:p-1].
\end{gathered}
\end{equation}
It follows that in this case $Z(n,k)=Z(n+1,k)$. The two cases that we just analyzed show that if $k\leq \ell(n)\leq p-1$ and $\eta(n)+\ell(n)\in[k:p-1]$, then $Z(n,k)=Z(n+1,k)$.

Suppose now that $k\leq \ell(n)\leq p-1$ and $\eta(n)+\ell(n)\geq p+k$. Then we know that $\kappa(n,j)=\kappa(n+1,j)$ for all $j$ except for $j=\eta(n)+\ell(n)-p$, and \eqref{relbounds:3}--\eqref{relbounds:4} are still valid. Notice that $k\leq \eta(n)+\ell(n)-p\leq \ell(n)-1\leq p-2$. Therefore in this case we have the following:
\begin{equation}\label{relbounds:7}
\begin{gathered}
\left[\varsigma(n,j):\kappa(n,j)\right]=\left[\varsigma(n+1,j):\kappa(n+1,j)\right]=[1:\lambda+1],\qquad j\in[k:\eta(n)+\ell(n)-p-1],\\
[\varsigma(n,j),\kappa(n,j)]=[1:\lambda],\qquad [\varsigma(n+1,j),\kappa(n+1,j)]=[1:\lambda+1],\qquad \mbox{for}\ j=\eta(n)+\ell(n)-p,\\
\left[\varsigma(n,j):\kappa(n,j)\right]=\left[\varsigma(n+1,j):\kappa(n+1,j)\right]=[1:\lambda],\qquad j\in[\eta(n)+\ell(n)-p+1:\ell(n)-1],\\
[\varsigma(n,j),\kappa(n,j)]=[0:\lambda],\qquad [\varsigma(n+1,j),\kappa(n+1,j)]=[1:\lambda],\qquad \mbox{for}\ j=\ell(n),\\
\left[\varsigma(n,j):\kappa(n,j)\right]=\left[\varsigma(n+1,j):\kappa(n+1,j)\right]=[0:\lambda],\qquad j\in[\ell(n)+1:p-1].
\end{gathered}
\end{equation}
This shows that if $k\leq \ell(n)\leq p-1$ and $\eta(n)+\ell(n)\geq p+k$, then $Z(n,k)=Z(n+1,k)$.

We assume now that $k\leq \ell(n)\leq p-1$ and $\eta(n)+\ell(n)\in[p:p+k-1]$. Then $\eta(n)+\ell(n)\notin[k:p-1]$ and $\eta(n)+\ell(n)-p\notin[k:p-1]$, hence $\kappa(n,j)=\kappa(n+1,j)$ for all $j\in[k:p-1]$. The relations \eqref{relbounds:3}--\eqref{relbounds:4} hold. Therefore we have
\begin{equation}\label{relbounds:8}
\begin{gathered}
\left[\varsigma(n,j):\kappa(n,j)\right]=\left[\varsigma(n+1,j):\kappa(n+1,j)\right]=[1:\lambda],\qquad j\in[k:\ell(n)-1],\\
[\varsigma(n,j),\kappa(n,j)]=[0:\lambda],\qquad [\varsigma(n+1,j),\kappa(n+1,j)]=[1:\lambda],\qquad \mbox{for}\ j=\ell(n),\\
\left[\varsigma(n,j):\kappa(n,j)\right]=\left[\varsigma(n+1,j):\kappa(n+1,j)\right]=[0:\lambda],\qquad j\in[\ell(n)+1:p-1].
\end{gathered}
\end{equation}
This shows that in this case $Z(n+1,k)=Z(n,k)-1$.

Finally let's assume that $\ell(n)=p$. Since $\ell(n+1)=0$, in this case we have $\varsigma(n,j)=1$ and $\varsigma(n+1,j)=0$ for all $j\in[k:p-1]$.
Writing $n=mp(p+1)+\eta(n)(p+1)+p$, we get
\[
\kappa(n,j)=\left\lfloor m+\frac{\eta(n)+p-j}{p}-\frac{1}{p(p+1)}\right\rfloor=\begin{cases}
\lambda+1, & j\in[0:\eta(n)-1],\\
\lambda, & j\in[\eta(n):p-1],
\end{cases}
\]
and
\[
\kappa(n+1,j)=\left\lfloor m+\frac{\eta(n)+1-j}{p}-\frac{1}{p(p+1)}\right\rfloor=\begin{cases}
\lambda, & j\in[0:\eta(n)],\\
\lambda-1, & j\in[\eta(n)+1:p-1].
\end{cases}
\]
If we assume that $\eta(n)\in[0:k-1]$, then
\begin{equation}\label{relbounds:9}
[\varsigma(n,j):\kappa(n,j)]=[1:\lambda],\quad
[\varsigma(n+1,j):\kappa(n+1,j)]=[0:\lambda-1],\qquad\mbox{for all}\ j\in[k:p-1],
\end{equation}
hence in this case $Z(n,k)=Z(n+1,k)$. If $\eta(n)\in[k:p-1]$, then
\begin{equation}\label{relbounds:10}
\begin{gathered}
\left[\varsigma(n,j):\kappa(n,j)\right]=[1:\lambda+1],\quad \left[\varsigma(n+1,j):\kappa(n+1,j)\right]=[0:\lambda],\qquad j\in[k:\eta(n)-1],\\
[\varsigma(n,j),\kappa(n,j)]=[1:\lambda],\qquad [\varsigma(n+1,j),\kappa(n+1,j)]=[0:\lambda],\qquad \mbox{for}\ j=\eta(n),\\
\left[\varsigma(n,j):\kappa(n,j)\right]=[1:\lambda],\quad\left[\varsigma(n+1,j):\kappa(n+1,j)\right]=[0:\lambda-1],\qquad j\in[\eta(n)+1:p-1],
\end{gathered}
\end{equation}
which shows that in this case $Z(n+1,k)=Z(n,k)+1$.

In virtue of \eqref{decompmodpp}--\eqref{eq:decomp:rho} we have $\ell(n)=\ell(n+p(p+1))$ and $\eta(n)=\eta(n+p(p+1))$, therefore the periodicity of $Z(n+1,k)-Z(n,k)$ follows from \eqref{eq:values:diffZnk:1}--\eqref{eq:values:diffZnk:2}.
\end{proof}

We illustrate in Table~\ref{table:values:Zn} the values of $Z(n+1,k)-Z(n,k)$ in the case $k=3$. It is worth noticing, as it follows from \eqref{eq:values:diffZnk:1}--\eqref{eq:values:diffZnk:2}, that if $k\geq 1$ and $\eta=\eta(n)=0$, then $Z(n+1,k)-Z(n,k)=0$; that is, the entries in the first row of the table of values of $Z(n+1,k)-Z(n,k)$, $k\geq 1$, are all zero.

\begin{table}[h]
\begin{center}
\begin{tikzpicture}[scale=0.8]
\draw [line width=2.1] (-6,5) -- (7.6,5);
\draw [line width=2.1] (-6,5) -- (-6,-5.5);
\draw [thin] (-7,4) -- (7.6,4);
\draw [thin] (-7,3) -- (7.6,3);
\draw [thin] (-7,2) -- (7.6,2);
\draw [thin] (-7,1) -- (7.6,1);
\draw [thin] (-7,0) -- (7.6,0);
\draw [thin] (-7,-1.5) -- (7.6,-1.5);
\draw [thin] (-7,-2.5) -- (7.6,-2.5);
\draw [thin] (-7,-3.5) -- (7.6,-3.5);
\draw [thin] (-7,-4.5) -- (7.6,-4.5);
\draw [thin] (-7,-5.5) -- (7.6,-5.5);
\draw [thin] (-4.8,6) -- (-4.8,-5.5);
\draw [thin] (-3.6,6) -- (-3.6,-5.5);
\draw [thin] (-2.4,6) -- (-2.4,-5.5);
\draw [thin] (-1.2, 6) -- (-1.2,-5.5);
\draw [thin] (0,6) -- (0,-5.5);
\draw [thin] (1.6,6) -- (1.6,-5.5);
\draw [thin] (2.8,6) -- (2.8,-5.5);
\draw [thin] (4,6) -- (4,-5.5);
\draw [thin] (5.2,6) -- (5.2,-5.5);
\draw [thin] (6.4,6) -- (6.4,-5.5);
\draw [thin] (7.6,6) -- (7.6, -5.5);
\draw (-5.4, 4.5) node [scale=1]{$0$};
\draw (-4.2,4.5) node [scale=1]{$0$};
\draw (-3, 4.5) node [scale=1]{$0$};
\draw (-1.8, 4.5) node [scale=1]{$0$};
\draw (-0.6,4.5) node [scale=1]{$0$};
\draw (-5.4,3.5) node [scale=1]{$0$};
\draw (-4.2, 3.5) node [scale=1]{$0$};
\draw (-3,3.5) node [scale=1]{$1$};
\draw (-1.8, 3.5) node [scale=1]{$0$};
\draw (-0.6,3.5) node [scale=1]{$0$};
\draw (-5.4,2.5) node [scale=1]{$0$};
\draw (-4.2, 2.5) node [scale=1]{$1$};
\draw (-3,2.5) node [scale=1]{$1$};
\draw (-1.8, 2.5) node [scale=1]{$0$};
\draw (-0.6,2.5) node [scale=1]{$0$};
\draw (-5.4,1.5) node [scale=1]{$1$};
\draw (-4.2, 1.5) node [scale=1]{$1$};
\draw (-3,1.5) node [scale=1]{$1$};
\draw (-1.8, 1.5) node [scale=1]{$0$};
\draw (-0.6,1.5) node [scale=1]{$0$};
\draw (-5.4,0.5) node [scale=1]{$1$};
\draw (-4.2, 0.5) node [scale=1]{$1$};
\draw (-3,0.5) node [scale=1]{$1$};
\draw (-1.8, 0.5) node [scale=1]{$0$};
\draw (-0.6,0.5) node [scale=1]{$0$};
\draw (-5.4,-2) node [scale=1]{$1$};
\draw (-4.2, -2) node [scale=1]{$1$};
\draw (-3, -2) node [scale=1]{$1$};
\draw (-1.8, -2) node [scale=1]{$0$};
\draw (-0.6, -2) node [scale=1]{$-1$};
\draw (-5.4,-3) node [scale=1]{$1$};
\draw (-4.2, -3) node [scale=1]{$1$};
\draw (-3, -3) node [scale=1]{$1$};
\draw (-1.8, -3) node [scale=1]{$-1$};
\draw (-0.6, -3) node [scale=1]{$-1$};
\draw (-5.4,-4) node [scale=1]{$1$};
\draw (-4.2, -4) node [scale=1]{$1$};
\draw (-3, -4) node [scale=1]{$0$};
\draw (-1.8, -4) node [scale=1]{$-1$};
\draw (-0.6, -4) node [scale=1]{$-1$};
\draw (-5.4,-5) node [scale=1]{$1$};
\draw (-4.2, -5) node [scale=1]{$0$};
\draw (-3, -5) node [scale=1]{$0$};
\draw (-1.8, -5) node [scale=1]{$-1$};
\draw (-0.6, -5) node [scale=1]{$0$};
\draw (2.2, 4.5) node [scale=1]{$0$};
\draw (3.4,4.5) node [scale=1]{$0$};
\draw (4.6, 4.5) node [scale=1]{$0$};
\draw (5.8, 4.5) node [scale=1]{$0$};
\draw (7,4.5) node [scale=1]{$0$};
\draw (2.2, 3.5) node [scale=1]{$0$};
\draw (3.4,3.5) node [scale=1]{$0$};
\draw (4.6, 3.5) node [scale=1]{$0$};
\draw (5.8, 3.5) node [scale=1]{$-1$};
\draw (7,3.5) node [scale=1]{$0$};
\draw (2.2, 2.5) node [scale=1]{$0$};
\draw (3.4,2.5) node [scale=1]{$0$};
\draw (4.6, 2.5) node [scale=1]{$-1$};
\draw (5.8, 2.5) node [scale=1]{$-1$};
\draw (7,2.5) node [scale=1]{$0$};
\draw (2.2, 1.5) node [scale=1]{$0$};
\draw (3.4,1.5) node [scale=1]{$-1$};
\draw (4.6, 1.5) node [scale=1]{$-1$};
\draw (5.8, 1.5) node [scale=1]{$-1$};
\draw (7,1.5) node [scale=1]{$1$};
\draw (2.2, 0.5) node [scale=1]{$-1$};
\draw (3.4,0.5) node [scale=1]{$-1$};
\draw (4.6, 0.5) node [scale=1]{$-1$};
\draw (5.8, 0.5) node [scale=1]{$0$};
\draw (7,0.5) node [scale=1]{$1$};
\draw (2.2, -2) node [scale=1]{$0$};
\draw (3.4,-2) node [scale=1]{$0$};
\draw (4.6, -2) node [scale=1]{$0$};
\draw (5.8, -2) node [scale=1]{$0$};
\draw (7, -2) node [scale=1]{$1$};
\draw (2.2, -3) node [scale=1]{$0$};
\draw (3.4,-3) node [scale=1]{$0$};
\draw (4.6, -3) node [scale=1]{$0$};
\draw (5.8, -3) node [scale=1]{$0$};
\draw (7, -3) node [scale=1]{$1$};
\draw (2.2, -4) node [scale=1]{$0$};
\draw (3.4,-4) node [scale=1]{$0$};
\draw (4.6, -4) node [scale=1]{$0$};
\draw (5.8, -4) node [scale=1]{$0$};
\draw (7, -4) node [scale=1]{$1$};
\draw (2.2, -5) node [scale=1]{$0$};
\draw (3.4,-5) node [scale=1]{$0$};
\draw (4.6, -5) node [scale=1]{$0$};
\draw (5.8, -5) node [scale=1]{$0$};
\draw (7, -5) node [scale=1]{$1$};
\draw (-5.4, 5.5) node [scale=1]{$0$};
\draw (-4.2,5.5) node [scale=1]{$1$};
\draw (-3, 5.5) node [scale=1]{$2$};
\draw (-1.8, 5.5) node [scale=1]{$3$};
\draw (-0.6,5.5) node [scale=1]{$4$};
\draw (2.2, 5.5) node [scale=1]{$p-4$};
\draw (3.4,5.5) node [scale=1]{$p-3$};
\draw (4.6, 5.5) node [scale=1]{$p-2$};
\draw (5.8, 5.5) node [scale=1]{$p-1$};
\draw (7, 5.5) node [scale=1]{$p$};
\draw (-6.6, 4.5) node [scale=1]{$0$};
\draw (-6.6, 3.5) node [scale=1]{$1$};
\draw (-6.6, 2.5) node [scale=1]{$2$};
\draw (-6.6, 1.5) node [scale=1]{$3$};
\draw (-6.6, 0.5) node [scale=1]{$4$};
\draw (-6.6, -2) node [scale=1]{$p-4$};
\draw (-6.6, -3) node [scale=1]{$p-3$};
\draw (-6.6, -4) node [scale=1]{$p-2$};
\draw (-6.6, -5) node [scale=1]{$p-1$};
\draw [line width=2] (-6,5) -- (-6,6);
\draw [line width=2] (-6,5) -- (-7,5);
\fill (-6.6,-1) circle (1.4pt);
\fill (-6.6,-0.8) circle (1.4pt);
\fill (-6.6, -0.6) circle (1.4pt);
\fill (-5.4,-1) circle (1.4pt);
\fill (-5.4,-0.8) circle (1.4pt);
\fill (-5.4, -0.6) circle (1.4pt);
\fill (-4.2,-1) circle (1.4pt);
\fill (-4.2,-0.8) circle (1.4pt);
\fill (-4.2, -0.6) circle (1.4pt);
\fill (-3,-1) circle (1.4pt);
\fill (-3,-0.8) circle (1.4pt);
\fill (-3, -0.6) circle (1.4pt);
\fill (-1.8,-1) circle (1.4pt);
\fill (-1.8,-0.8) circle (1.4pt);
\fill (-1.8, -0.6) circle (1.4pt);
\fill (-0.6,-1) circle (1.4pt);
\fill (-0.6,-0.8) circle (1.4pt);
\fill (-0.6, -0.6) circle (1.4pt);
\fill (2.2,-1) circle (1.4pt);
\fill (2.2,-0.8) circle (1.4pt);
\fill (2.2, -0.6) circle (1.4pt);
\fill (3.4,-1) circle (1.4pt);
\fill (3.4,-0.8) circle (1.4pt);
\fill (3.4, -0.6) circle (1.4pt);
\fill (4.6,-1) circle (1.4pt);
\fill (4.6,-0.8) circle (1.4pt);
\fill (4.6, -0.6) circle (1.4pt);
\fill (5.8,-1) circle (1.4pt);
\fill (5.8,-0.8) circle (1.4pt);
\fill (5.8, -0.6) circle (1.4pt);
\fill (7,-1) circle (1.4pt);
\fill (7,-0.8) circle (1.4pt);
\fill (7, -0.6) circle (1.4pt);
\fill (0.6, 5.5) circle (1.4pt);
\fill (0.8, 5.5) circle (1.4pt);
\fill (1, 5.5) circle (1.4pt);
\fill (0.6, 4.5) circle (1.4pt);
\fill (0.8, 4.5) circle (1.4pt);
\fill (1, 4.5) circle (1.4pt);
\fill (0.6, 3.5) circle (1.4pt);
\fill (0.8, 3.5) circle (1.4pt);
\fill (1, 3.5) circle (1.4pt);
\fill (0.6, 2.5) circle (1.4pt);
\fill (0.8, 2.5) circle (1.4pt);
\fill (1, 2.5) circle (1.4pt);
\fill (0.6, 1.5) circle (1.4pt);
\fill (0.8, 1.5) circle (1.4pt);
\fill (1, 1.5) circle (1.4pt);
\fill (0.6, 0.5) circle (1.4pt);
\fill (0.8, 0.5) circle (1.4pt);
\fill (1, 0.5) circle (1.4pt);
\fill (0.6, -2) circle (1.4pt);
\fill (0.8, -2) circle (1.4pt);
\fill (1, -2) circle (1.4pt);
\fill (0.6, -3) circle (1.4pt);
\fill (0.8, -3) circle (1.4pt);
\fill (1, -3) circle (1.4pt);
\fill (0.6, -4) circle (1.4pt);
\fill (0.8, -4) circle (1.4pt);
\fill (1, -4) circle (1.4pt);
\fill (0.6, -5) circle (1.4pt);
\fill (0.8, -5) circle (1.4pt);
\fill (1, -5) circle (1.4pt);
\draw (0.8,6.3) node [scale=1]{$\ell$};
\draw (-7.5,-0.8) node [scale=1]{$\eta$};
\end{tikzpicture}
\end{center}
\caption{The values of $Z(n+1,3)-Z(n,3)$ that correspond to the values of $\rho(n)=\eta(p+1)+\ell$ with $0\leq \eta\leq p-1$ and $0\leq \ell\leq p$.}
\label{table:values:Zn}
\end{table}

We now state a useful lemma that will be applied in the following section.

\begin{lemma}
Let $k\in[1:p-1]$. Then we have the following:
\begin{equation}\label{eq:rel:diffZnk}
Z(n+1,k)-Z(n,k)\leq Z(n+1,k-1)-Z(n,k-1)+\begin{cases}
0, & \mathrm{if} \ \ell(n)\in[0:k-2],\\[0.5em]
1, & \mathrm{if} \ \ell(n)=k-1,\\[0.5em]
0, & \mathrm{if} \ \ell(n)\in[k:p].
\end{cases}
\end{equation}
\end{lemma}
\begin{proof}
To prove the first relation in \eqref{eq:rel:diffZnk}, assume that $\ell(n)\in[0:k-2]$. If $\eta(n)+\ell(n)\in[k:p-1]$, then in virtue of \eqref{eq:values:diffZnk:1} we have $Z(n+1,k)-Z(n,k)=1$. Now, since $\ell(n)\in[0:k-2]$ and $\eta(n)+\ell(n)\in[k:p-1]\subset[k-1:p-1]$, applying the second relation in \eqref{eq:values:diffZnk:1} for $k-1$ we obtain $Z(n+1,k-1)-Z(n,k-1)=1$. Now assume that $\eta(n)+\ell(n)\notin[k:p-1]$. Then $Z(n+1,k)-Z(n,k)=0$ and we have two possibilities, namely $\eta(n)+\ell(n)\notin[k-1:p-1]$ or $\eta(n)+\ell(n)=k-1$. In the former case, we have $Z(n+1,k-1)-Z(n,k-1)=0$, and in the later case $Z(n+1,k-1)-Z(n,k-1)=1$. This justifies the first relation in \eqref{eq:rel:diffZnk}.

Now assume that $\ell(n)=k-1$. We distinguish again the two alternatives $\eta(n)+\ell(n)\in[k:p-1]$ and $\eta(n)+\ell(n)\notin[k:p-1]$. In the first one we have $Z(n+1,k)-Z(n,k)=1$, and because $\ell(n)\in[k-1:p-1]$ and $\eta(n)+\ell(n)\notin[k:p-1]$ we get from \eqref{eq:values:diffZnk:1} that $Z(n+1,k-1)-Z(n,k-1)=0$, so the claim holds. In the second case we have $Z(n+1,k)-Z(n,k)=0$, so the claim holds trivially since $Z(n+1,k-1)-Z(n,k-1)\in\{-1,0,1\}$. This proves the second relation in \eqref{eq:rel:diffZnk}.

Suppose now that $\ell(n)\in[k:p-1]$. If $\eta(n)+\ell(n)\in[p:p+k-1]$ then $Z(n+1,k)-Z(n,k)=-1$ and the inequality is trivially valid. If $\eta(n)+\ell(n)\notin[p:p+k-1]$, then $Z(n+1,k)-Z(n,k)=0$. The assumptions imply that $\ell(n)\in[k-1:p-1]$ and $\eta(n)+\ell(n)\notin[p:p+k-2]$, hence $Z(n+1,k-1)-Z(n,k-1)=0$, and the claim follows.

Finally, if $\ell(n)=p$, inequality follows immediately from \eqref{eq:values:diffZnk:2}.
\end{proof}

In what follows, we shall use the notations $\mathrm{sign}(f,\Delta)$ ($\mathrm{sign}(\nu,\Delta)$) to indicate the sign of the function $f$ (measure $\nu$) on the interval $\Delta$, and $\Delta_{k}$ shall denote the interval $[a_{k},b_{k}]$. Recall that $\varepsilon_{n,k}=\mathrm{sign}(\nu_{n,k},\Delta_{k})$, where $\nu_{n,k}$ is the measure defined in \eqref{def:measnunk}. We continue using the notations \eqref{decompmodpp}--\eqref{eq:decomp:rho}.

\begin{lemma}\label{lem:ratioepsilon}
Let $k\in[0:p-1]$ be fixed. Then, as a function of $n$, the expression $\frac{\varepsilon_{n+1,k}}{\varepsilon_{n,k}}$ is periodic with period $p(p+1)$. More precisely, we have
\begin{equation}\label{eq:ratioepsilon}
\frac{\varepsilon_{n+1,k}}{\varepsilon_{n,k}}=(-1)^{Z(n+1,2\lc k/2 \rc)-Z(n,2\lc k/2 \rc)+\theta(n,k)},
\end{equation}
where
\begin{equation}\label{def:thetank}
\theta(n,k):=\begin{cases}
1, & \mbox{if} \ \ell(n)\in[0:k-1],\\
0, & \mbox{if} \ \ell(n)\in[k+1:p-1],\\
1, & \mbox{if} \ \ell(n)=k,\,\,k\,\,\mbox{odd},\\
0, & \mbox{if} \ \ell(n)=k,\,\,k\,\,\mbox{even},\\
1, & \mbox{if} \ \ell(n)=p.
\end{cases}
\end{equation}
\end{lemma}
\begin{proof}
From \eqref{relpositint} and the definition of the measure $\sigma_{n,k}$ in \eqref{varymeas:sigma} we deduce that
\[
\mathrm{sign}(\sigma_{n,k},\Delta_{k})=\begin{cases}
1 & \mbox{for}\,\,k\,\,\mbox{even},\\
1 & \mbox{for}\,\,k\,\,\mbox{odd},\,\,\ell(n)\leq k,\\
-1 & \mbox{for}\,\,k\,\,\mbox{odd},\,\, k<\ell(n).
\end{cases}
\]
Since $P_{n,k}$ is a monic polynomial of degree $Z(n,k)$ and its zeros are located in $\Delta_{k}$, we have
\[
\mathrm{sign}(P_{n,k-1} P_{n,k+1},\Delta_{k})=\begin{cases}
1 & \mbox{for}\,\,k\,\,\mbox{even},\\
(-1)^{Z(n,k-1)+Z(n,k+1)} & \mbox{for}\,\,k\,\,\mbox{odd}.
\end{cases}
\]
In view of \eqref{intrep:hnk} we also obtain
\[
\mathrm{sign}(h_{n,k},\Delta_{k})=\begin{cases}
\varepsilon_{n,k-1} & \mbox{for}\,\,k\,\,\mbox{even},\\
\varepsilon_{n,k-1} & \mbox{for}\,\,k\,\,\mbox{odd},\,\,\ell(n)<k,\\
(-1)\varepsilon_{n,k-1} & \mbox{for}\,\,k\,\,\mbox{odd},\,\,k\leq \ell(n).
\end{cases}
\]
From the above sign formulas and \eqref{def:measnunk} we conclude that for each $k=1,\ldots,p-1,$
\begin{equation}\label{eq:relepsilons}
\epsilon_{n,k}=\mathrm{sign}(\nu_{n,k},\Delta_{k})=\begin{cases}
\varepsilon_{n,k-1} & \mbox{for}\,\,k\,\,\mbox{even},\\
\varepsilon_{n,k-1} (-1)^{Z(n,k-1)+Z(n,k+1)} & \mbox{for}\,\,k\,\,\mbox{odd},\,\,\ell(n)\neq k,\\
\varepsilon_{n,k-1} (-1)^{Z(n,k-1)+Z(n,k+1)+1} & \mbox{for}\,\,k\,\,\mbox{odd},\,\,\ell(n)=k.
\end{cases}
\end{equation}
Notice also that $\varepsilon_{n,0}=1$ for all $n$.

A careful iterative application of \eqref{eq:relepsilons} gives the following formula, valid for all $k\in[0:p-1]$:
\begin{equation}\label{eq:formulaenk}
\varepsilon_{n,k}=\begin{cases}
(-1)^{Z(n,0)+Z(n,2\lc k/2 \rc)} & \mbox{for}\,\,\ell(n)\,\,\mbox{even},\\
(-1)^{Z(n,0)+Z(n,2\lc k/2 \rc)} & \mbox{for}\,\,\ell(n)\,\,\mbox{odd},\,\,k<\ell(n),\\
(-1)^{Z(n,0)+Z(n,2\lc k/2 \rc)+1} & \mbox{for}\,\,\ell(n)\,\,\mbox{odd},\,\,\ell(n)\leq k.
\end{cases}
\end{equation}
Recall also that by convention $Z(n,p)=0$.

Suppose that $\ell(n)\in[0:p-1]$ and $\ell(n)$ is even. Then $\ell(n+1)=\ell(n)+1$ is odd and $Z(n+1,0)=Z(n,0)$, therefore from \eqref{eq:formulaenk} we obtain
\begin{equation}\label{eq:epsilon:1}
\frac{\varepsilon_{n+1,k}}{\varepsilon_{n,k}}
=\frac{(-1)^{Z(n+1,0)+Z(n+1,2\lc k/2 \rc)+\theta_{1}(n,k)}}{(-1)^{Z(n,0)+Z(n,2\lc k/2 \rc)}}=(-1)^{Z(n+1,2\lc k/2 \rc)-Z(n,2\lc k/2 \rc)+\theta_{1}(n,k)},
\end{equation}
where
\[
\theta_{1}(n,k):=\begin{cases}
0, & \mbox{if} \ 0\leq k\leq \ell(n),\\
1, & \mbox{if} \ \ell(n)+1\leq k\leq p-1.
\end{cases}
\]
If $\ell(n)\in[0:p-1]$ and $\ell(n)$ is odd, then $\ell(n+1)=\ell(n)+1$ is even and again $Z(n+1,0)=Z(n,0)$. Hence
\begin{equation}\label{eq:epsilon:2}
\frac{\varepsilon_{n+1,k}}{\varepsilon_{n,k}}=\frac{(-1)^{Z(n+1,0)+Z(n+1,2\lc k/2 \rc)}}{(-1)^{Z(n,0)+Z(n,2\lc k/2 \rc)+\theta_{2}(n,k)}}=(-1)^{Z(n+1,2\lc k/2 \rc)-Z(n,2\lc k/2 \rc)+\theta_{2}(n,k)},
\end{equation}
where
\[
\theta_{2}(n,k):=\begin{cases}
0, & \mbox{if} \ 0\leq k\leq \ell(n)-1,\\
1, & \mbox{if} \ \ell(n)\leq k\leq p-1.
\end{cases}
\]
From \eqref{eq:epsilon:1} and \eqref{eq:epsilon:2} we deduce that \eqref{eq:ratioepsilon} holds in the case $\ell(n)\in[0:p-1]$. The justification of \eqref{eq:ratioepsilon} in the case $\ell(n)=p$ is done similarly and it is left to the reader.
\end{proof}

\section{Interlacing property of the zeros of $P_{n,k}$}\label{sec:interlacing}

In this section we prove the interlacing property of the zeros of the polynomials $P_{n,k}$. Recall that by definition, $P_{n,k}$ is the monic polynomial whose zeros are the zeros of $\psi_{n,k}$ in $(a_{k},b_{k})$. These zeros are all simple and there are $Z(n,k)$ of them, cf. Proposition~\ref{prop:lpsinksummary}. The interlacing property will be derived from a series of lemmas that will be proved first. We remark that in the case $k=0$, the interlacing property of the zeros of $P_{n,0}$ is already a consequence of property 5) from Proposition~\ref{prop:Qnsummary}.

In this section we will continue using the notation in \eqref{decompmodpp}--\eqref{eq:decomp:rho}.

\subsection{Auxiliary results}

\begin{lemma}\label{lemma:interlac:1}
Assume that $A, B\in\mathbb{R}$ with $|A|+|B|>0$. For $k\in[0:p-1]$ and $n\geq 0$ integers, let
\begin{equation}\label{def:Gnk}
G_{n,k}(z):=\begin{cases}
A \psi_{n,k}(z)+B \psi_{n+1,k}(z), & \ell(n)\neq p,\\
A z\psi_{n,k}(z)+B \psi_{n+1,k}(z), & \ell(n)=p.
\end{cases}
\end{equation}
Then the following properties hold:
\begin{itemize}
\item[1)] If $\ell(n)\in[0:k-1]$, then $G_{n,k}$ has at most $Z(n,k)+1$ zeros in $\mathbb{C}\setminus([a_{k-1},b_{k-1}]\cup\{0\})$, counting multiplicities.
\item[2)] If $\ell(n)\in[k:p-1]$, then $G_{n,k}$ has at most $Z(n,k)$ zeros in $\mathbb{C}\setminus([a_{k-1},b_{k-1}]\cup\{0\})$, counting multiplicities.
\item[3)] If $\ell(n)=p$, then $G_{n,k}$ has at most $Z(n,k)+1$ zeros in $\mathbb{C}\setminus([a_{k-1},b_{k-1}]\cup\{0\})$, counting multiplicities.
\end{itemize}
\end{lemma}
\begin{proof}
The proof is done by induction on $k$. First, since $\psi_{n,0}=P_{n,0}$, in the case $k=0$ we have
\[
G_{n,0}(z)=\begin{cases}
A P_{n,0}(z)+B P_{n+1,0}(z), & \ell(n)\neq p, \\
A z P_{n,0}(z)+B P_{n+1,0}(z), & \ell(n)=p.
\end{cases}
\]
If $\ell(n)\neq p$, then $G_{n,0}$ is a polynomial of degree at most $Z(n,0)=\deg(P_{n,0})=\deg(P_{n+1,0})$, and if $\ell(n)=p$, then $G_{n,0}$ is a polynomial of degree at most $Z(n,0)+1=\deg(P_{n,0})+1=\deg(P_{n+1,0})$. Moreover, $G_{n,0}$ cannot be identically zero, as it easily follows from \eqref{threetermrecsecondkindpequena1}--\eqref{threetermrecsecondkindpequena2} or from the fact that the zeros of $P_{n,0}$ and $P_{n+1,0}$ do not coincide. Hence the result holds in the case $k=0$.

Assume that the result holds for $k-1$ but it does not hold for $k$, where $1\leq k\leq p-1$. So suppose first that $\ell(n)\in[0:k-1]$ and $G_{n,k}(z)=A\psi_{n,k}(z)+B\psi_{n+1,k}(z)$ has at least $Z(n,k)+2$ zeros in $\mathbb{C}\setminus([a_{k-1},b_{k-1}]\cup\{0\})$, counting multiplicities.
Note that $G_{n,k}$ is analytic in the complement of $[a_{k-1},b_{k-1}]$ and satisfies $G_{n,k}(\overline{z})=\overline{G_{n,k}(z)}$, therefore the complex non-real zeros of $G_{n,k}$, if any, must come in conjugate pairs. Thus, we can construct a monic polynomial $L_{n,k}$ with real coefficients and degree at least $Z(n,k)+2$ whose zeros are zeros of $G_{n,k}$ in $\mathbb{C}\setminus([a_{k-1},b_{k-1}]\cup\{0\})$.

Let's assume that in fact $\ell(n)<k-1$ (the case $\ell(n)=k-1$ will be analyzed later). Then according to \eqref{varymeas:sigma} and \eqref{recurrenceforpsipequena}, we have
\begin{equation}\label{eq:rep:psink:1}
\psi_{n,k}(z)=z\int_{a_{k-1}}^{b_{k-1}}\frac{\psi_{n,k-1}(\tau)}{z-\tau} d\sigma_{k-1}^{*}(\tau),
\end{equation}
and since $\ell(n+1)=\ell(n)+1\leq k-1$, we also have
\begin{equation}\label{eq:rep:psink:2}
\psi_{n+1,k}(z)=z\int_{a_{k-1}}^{b_{k-1}}\frac{\psi_{n+1,k-1}(\tau)}{z-\tau} d\sigma_{k-1}^{*}(\tau),
\end{equation}
hence
\begin{equation}\label{rep:Gnk:1}
G_{n,k}(z)=z\int_{a_{k-1}}^{b_{k-1}}\frac{A\psi_{n,k-1}(\tau)+B\psi_{n+1,k-1}(\tau)}{z-\tau} d\sigma_{k-1}^{*}(\tau).
\end{equation}
It follows that
\[\frac{G_{n,k}(z)}{zL_{n,k}(z)} \in \mathcal{H}(\mathbb{C}\setminus [a_{k-1},b_{k-1}]).\]

Let us analyze the order of the zero at infinity. By Proposition \ref{prop:lpsinksummary}, we have
\[\psi_{n,k}(z) = O(z^{-Z(n,k-1)+Z(n,k)}), \qquad \psi_{n+1,k}(z) = O(z^{-Z(n+1,k-1)+Z(n+1,k)}).
\]
Taking this and the degree of $zL_{n,k}(z)$ into account, it follows that
\[
\frac{G_{n,k}(z)}{zL_{n,k}(z)} = O(z^{- Z(n,k-1) -3}) + O(z^{- Z(n+1,k-1) + Z(n+1,k)- Z(n,k) -3}).
\]
Since $\ell(n)\in[0:k-2]$, by \eqref{eq:rel:diffZnk} we obtain that $-Z(n+1,k-1)+Z(n+1,k)-Z(n,k)\leq -Z(n,k-1)$, hence we have
\begin{equation}\label{est:Gnk:1}
\frac{G_{n,k}(z)}{zL_{n,k}(z)} = O(z^{- Z(n,k-1) -3}).
\end{equation}
We consider now a simple Jordan curve $\gamma$ that surrounds $[a_{k-1},b_{k-1}]$ and leaves the zeros of $L_{n,k}$ outside. Then from \eqref{est:Gnk:1} and \eqref{rep:Gnk:1} we deduce that for $j=0,\ldots,Z(n,k-1)+1$,
\begin{align*}
0= & \frac{1}{2\pi i}\int_{\gamma}\frac{G_{n,k}(z)}{z L_{n,k}(z)}\,z^{j}\,dz\\
= & \frac{1}{2\pi i}\int_{\gamma}\frac{z^{j}}{L_{n,k}(z)}\left(\int_{a_{k-1}}^{b_{k-1}}\frac{A\psi_{n,k-1}(\tau)+B\psi_{n+1,k-1}(\tau)}{z-\tau} d\sigma_{k-1}^{*}(\tau)\right) dz\\
= & \int_{a_{k-1}}^{b_{k-1}}(A\psi_{n,k-1}(\tau)+B\psi_{n+1,k-1}(\tau))\,\frac{\tau^{j}}{L_{n,k}(\tau)}\,d\sigma_{k-1}^{*}(\tau),
\end{align*}
where we applied Cauchy's theorem and integral formula and Fubini's theorem. This implies that $G_{n,k-1}=A\psi_{n,k-1}+B\psi_{n+1,k-1}$ has $Z(n,k-1)+2$ zeros with odd multiplicity in $(a_{k-1},b_{k-1})$, which contradicts statement 1) for $k-1$ (recall that $\ell(n)\in[0:k-2]$ at this point). This finishes the analysis in the case $\ell(n)<k-1$.

Assume now that $\ell(n)=k-1$. Then \eqref{eq:rep:psink:1} is valid and since $\ell(n+1)=\ell(n)+1=k$, \eqref{eq:rep:psink:2} is replaced by
\[
\psi_{n+1,k}(z)=\int_{a_{k-1}}^{b_{k-1}}\frac{\psi_{n+1,k-1}(\tau)}{z-\tau}\,\tau\,d\sigma_{k-1}^{*}(\tau),
\]
hence in this case
\begin{equation}\label{rep:Gnk:2}
G_{n,k}(z)=A\psi_{n,k}(z)+B\psi_{n+1,k}(z)=\int_{a_{k-1}}^{b_{k-1}}\frac{A\, z\,\psi_{n,k-1}(\tau)+B\,\tau\,\psi_{n+1,k-1}(\tau)}{z-\tau}\,d\sigma_{k-1}^{*}(\tau).
\end{equation}

In virtue of \eqref{decayinfpsink} and \eqref{rel:NnkZnk} we have the following estimates at infinity:
\begin{align*}
\psi_{n,k}(z) & =O(z^{-N(n,k)})=O(z^{-Z(n,k-1)+Z(n,k)}),\\
\psi_{n+1,k}(z) & =O(z^{-N(n+1,k)})=O(z^{-Z(n+1,k-1)+Z(n+1,k)-1}),
\end{align*}
hence
\[
\frac{G_{n,k}(z)}{L_{n,k}(z)}=O(z^{-Z(n,k-1)-2})+O(z^{-Z(n+1,k-1)+Z(n+1,k)-Z(n,k)-3}).
\]
Since $\ell(n)=k-1$, applying \eqref{eq:rel:diffZnk} we see that $-Z(n+1,k-1)+Z(n+1,k)-Z(n,k)\leq -Z(n,k-1)+1$, so we obtain
\begin{equation}\label{est:Gnk:5}
\frac{G_{n,k}(z)}{L_{n,k}(z)}=O(z^{-Z(n,k-1)-2}).
\end{equation}

Taking a curve $\gamma$ as before, we deduce from \eqref{rep:Gnk:2} and \eqref{est:Gnk:5} that for $j=0,\ldots,Z(n,k-1)$,
\begin{align*}
0= & \frac{1}{2\pi i}\int_{\gamma}\frac{G_{n,k}(z)}{L_{n,k}(z)}\,z^{j}\,dz\\
= & \frac{1}{2\pi i}\int_{\gamma}\frac{z^{j}}{L_{n,k}(z)}\left(\int_{a_{k-1}}^{b_{k-1}}\frac{A\, z\,\psi_{n,k-1}(\tau)+B\,\tau\,\psi_{n+1,k-1}(\tau)}{z-\tau} d\sigma_{k-1}^{*}(\tau)\right) dz\\
= & \int_{a_{k-1}}^{b_{k-1}}(A\psi_{n,k-1}(\tau)+B\psi_{n+1,k-1}(\tau))\,\tau^{j}\,\frac{\tau\, d\sigma_{k-1}^{*}(\tau)}{L_{n,k}(\tau)}.
\end{align*}
This implies that $G_{n,k-1}=A\psi_{n,k-1}+B\psi_{n+1,k-1}$ has at least $Z(n,k-1)+1$ zeros with odd multiplicity in $(a_{k-1},b_{k-1})$, contradicting statement 2) for $k-1$. This concludes the proof of statement 1) for $k$.

The proofs of 2) and 3) proceed in a similar way.

Assume that $\ell(n)\in[k:p-1]$ and $G_{n,k}(z)$ has at least $Z(n,k)+1$ zeros in $\mathbb{C}\setminus([a_{k-1},b_{k-1}]\cup\{0\})$, counting multiplicities. As it was done before, we can take a monic polynomial $L_{n,k}$ with real coefficients and degree at least $Z(n,k)+1$ whose zeros are zeros of $G_{n,k}$ in $\mathbb{C}\setminus([a_{k-1},b_{k-1}]\cup\{0\})$. According to \eqref{recurrenceforpsipequena} and \eqref{varymeas:sigma}, in this case we have
\[
G_{n,k}(z)=A\psi_{n,k}(z)+B\psi_{n+1,k}(z)=\int_{a_{k-1}}^{b_{k-1}}\frac{A\psi_{n,k-1}(\tau)+B\psi_{n+1,k-1}(\tau)}{z-\tau}\,\tau\,d\sigma_{k-1}^{*}(\tau).
\]
Since $\ell(n)\geq k$ and $\ell(n+1)=\ell(n)+1>k$, the following estimates hold as $z$ approaches infinity:
\begin{align*}
\psi_{n,k}(z) & =O(z^{-N(n,k)})=O(z^{-Z(n,k-1)+Z(n,k)-1}),\\
\psi_{n+1,k}(z) & =O(z^{-N(n+1,k)})=O(z^{-Z(n+1,k-1)+Z(n+1,k)-1}),
\end{align*}
hence we have
\begin{equation}
\frac{G_{n,k}(z)}{L_{n,k}(z)}=O(z^{-Z(n,k-1)-2})+O(z^{-Z(n+1,k-1)+Z(n+1,k)-Z(n,k)-2})=O(z^{-Z(n,k-1)-2}),
\end{equation}
where in the second equality we applied the third relation in \eqref{eq:rel:diffZnk}. This implies, as it was done before, that $G_{n,k-1}=A\psi_{n,k-1}+B\psi_{n+1,k-1}$ has at least $Z(n,k-1)+1$ zeros in $(a_{k-1},b_{k-1})$, contradicting statement 2) for $k-1$.

Finally, assume that $\ell(n)=p$, and assume that $G_{n,k}(z)=A z\psi_{n,k}(z)+B\psi_{n+1,k}(z)$ has at least $Z(n,k)+2$ zeros in $\mathbb{C}\setminus([a_{k-1},b_{k-1}]\cup\{0\})$, counting multiplicities. Let $L_{n,k}$ be a polynomial with real coefficients and degree at least $Z(n,k)+2$ whose zeros are zeros of $G_{n,k}$ in $\mathbb{C}\setminus([a_{k-1},b_{k-1}]\cup\{0\})$.

Since $\ell(n)=p>k$ and $\ell(n+1)=0<k$, applying \eqref{recurrenceforpsipequena} and \eqref{varymeas:sigma} we obtain
\[
\psi_{n,k}(z)=\int_{a_{k-1}}^{b_{k-1}}\frac{\psi_{n,k-1}(\tau)}{z-\tau}\,\tau\,d\sigma_{k-1}^{*}(\tau),\qquad \psi_{n+1,k}(z)=z\int_{a_{k-1}}^{b_{k-1}}\frac{\psi_{n+1,k-1}(\tau)}{z-\tau}\,d\sigma_{k-1}^{*}(\tau),
\]
therefore
\[
G_{n,k}(z)=A z\psi_{n,k}(z)+B\psi_{n+1,k}(z)=z\int_{a_{k-1}}^{b_{k-1}}\frac{A\tau\psi_{n,k-1}(\tau)+B\psi_{n+1,k-1}(\tau)}{z-\tau}\,d\sigma_{k-1}^{*}(\tau)
\]
and the function $\frac{G_{n,k}(z)}{zL_{n,k}(z)}$ is analytic outside $[a_{k-1},b_{k-1}]$. Applying \eqref{decayinfpsink} and \eqref{rel:NnkZnk} we obtain
\begin{align*}
\psi_{n,k}(z) & =O(z^{-N(n,k)})=O(z^{-Z(n,k-1)+Z(n,k)-1}),\\
\psi_{n+1,k}(z) & =O(z^{-N(n+1,k)})=O(z^{-Z(n+1,k-1)+Z(n+1,k)}),
\end{align*}
which implies that
\[
\frac{G_{n,k}(z)}{z L_{n,k}(z)}=O(z^{-Z(n,k-1)-3})+O(z^{-Z(n+1,k-1)+Z(n+1,k)-Z(n,k)-3})=O(z^{-Z(n,k-1)-3}),
\]
where in the last equality we have used \eqref{eq:rel:diffZnk}. This easily implies, as shown before,
that the function $A z\psi_{n,k-1}(z)+B\psi_{n+1,k-1}(z)$ has at least $Z(n,k-1)+2$ zeros with odd multiplicity in $(a_{k-1},b_{k-1})$, contradicting statement 3) for $k-1$.
\end{proof}

The following lemma is Corollary 2.15 from \cite{LopMin}, and it was obtained as an application of an AT system property satisfied by the Cauchy transforms of the measures $\mu_{k,j}$, see Section 2.4 in \cite{LopMin}.

\begin{lemma}\label{lemma:orthosystem} Let $k$, $r$ be integers such that
$0\leq k\leq r\leq  p-1$. Let $\{d_j\}_{j=k}^{p-1}$ be a finite sequence of
nonnegative integers such that
\[
d_k\geq d_{k+1}\geq \cdots \geq d_{r}\geq d_{r+1}-1\geq d_{r+2}-1\geq
\cdots\geq d_{p-1}-1.
\]
Suppose $F\not\equiv 0$ is a function analytic and real-valued on $[a_k,b_k]$,
satisfying the orthogonality conditions
\begin{align}\label{orthocondgeneral1}
 \int_{a_k}^{b_k}F(\tau)\tau^{s+\delta}d\mu_{k,j}(\tau)=0,\quad 0\leq s\leq d_j-1,\quad
k\leq j\leq r,
\end{align}
\begin{align}\label{orthocondgeneral2}
 \int_{a_k}^{b_k}F(\tau)\tau^{s}d\mu_{k,j}(\tau)=0,\quad 0\leq s\leq d_j-1,\quad r<
j\leq p-1,
\end{align}
where the constant $\delta=1$ if $r<p-1$ and $d_{r+1}=d_r+1$, otherwise
$\delta$ could be taken to be either $1$ or $0$. Then, $F$ has at least
\[
N:=\sum_{j=k}^{p-1}d_j
\]
zeros of odd multiplicity in $(a_{k},b_k)$.
\end{lemma}

Lemma~\ref{lemma:orthosystem} will be repeatedly applied in the proof of the following result, which complements Lemma~\ref{lemma:interlac:1}.

\begin{lemma}\label{lemma:interlac:2}
 Assume that $A, B\in\mathbb{R}$, $|A|+|B|>0$, and let $k\in[0:p-1]$ and $n\geq 0$ be integers. Then the function $G_{n,k}$ defined in \eqref{def:Gnk} satisfies the following properties:
\begin{itemize}
\item[1)] If $\ell(n)\in[0:k-1]$, then $G_{n,k}$ has at least $Z(n,k)$ zeros with odd multiplicity in $(a_{k},b_{k})$.
\item[2)] If $\ell(n)\in[k:p-1]$, then $G_{n,k}$ has at least $Z(n,k)-1$ zeros with odd multiplicity in $(a_{k},b_{k})$.
\item[3)] If $\ell(n)=p$, then $G_{n,k}$ has at least $Z(n,k)$ zeros with odd multiplicity in $(a_{k},b_{k})$.
\end{itemize}
\end{lemma}
\begin{proof}
Assume that $\ell(n)<k$ and $\eta(n)+\ell(n)\in[k:p-1]$. Then the relations \eqref{relbounds:1}--\eqref{relbounds:2} hold, and recall that if $j=\eta(n)+\ell(n)$ then $\kappa(n+1,k)=\kappa(n,k)+1$. As a consequence, both functions $\psi_{n,k}$ and $\psi_{n+1,k}$ satisfy the same orthogonality conditions \eqref{orthogredPsink}. Hence
\begin{equation}\label{int:Gnk:1}
\int_{a_{k}}^{b_{k}}G_{n,k}(\tau)\,\tau^{s}\,d\mu_{k,j}(\tau)=0,\quad
\varsigma(n,j)\leq s\leq \kappa(n,j),\quad k\leq j\leq p-1.
\end{equation}
In this case $\varsigma(n,j)=0$ for all $j\in[k:p-1]$ and the sequence $\{\kappa(n,j)\}_{j=k}^{p-1}$ is non-increasing. So we can apply Lemma~\ref{lemma:orthosystem} to $F=G_{n,k}$, taking $\delta=0$ in \eqref{orthocondgeneral1}, $r=k$ (for example), and $d_{j}=\kappa(n,j)$ for all $j\in[k:p-1]$. It follows that $G_{n,k}$ has at least $Z(n,k)$ zeros with odd multiplicity in $(a_{k},b_{k})$.

Assume now that $\ell(n)<k$ and $\eta(n)+\ell(n)\notin[k:p-1]$. In this case we have $\varsigma(n,j)=\varsigma(n+1,j)=0$ for all $j\in[k:p-1]$, and as it was observed in the proof of Lemma~\ref{lem:valuesdiffZnk}, we also have $\kappa(n,j)=\kappa(n+1,j)$ for all $j$. Hence \eqref{int:Gnk:1} holds and applying Lemma~\ref{lemma:orthosystem} to $F=G_{n,k}$ as before we obtain that $G_{n,k}$ has at least $Z(n,k)$ zeros with odd multiplicity in $(a_{k},b_{k})$. This finishes the proof of part 1).

Suppose that $\ell(n)\in[k:p-1]$ and $\eta(n)+\ell(n)\in[k:p-1]$, and assume additionally for the moment that $\eta(n)=0$. We then deduce from \eqref{orthogredPsink} and \eqref{relbounds:5} that the function $G_{n,k}=A\psi_{n,k}+B\psi_{n+1,k}$ satisfies the following orthogonality conditions. For each $j\in[k:\ell(n)-1]$,
\[
\int_{a_{k}}^{b_{k}}G_{n,k}(\tau)\tau^{s+1} d\mu_{k,j}(\tau)=0,\qquad 0\leq s\leq \lambda-1,
\]
for $j=\ell(n)$,
\[
\int_{a_{k}}^{b_{k}}G_{n,k}(\tau)\tau^{s+1} d\mu_{k,j}(\tau)=0,\qquad 0\leq s\leq \lambda-2,
\]
and for each $j\in[\ell(n)+1:p-1]$ we have
\[
\int_{a_{k}}^{b_{k}}G_{n,k}(\tau)\tau^{s} d\mu_{k,j}(\tau)=0,\qquad 0\leq s\leq \lambda-1.
\]
If we apply Lemma~\ref{lemma:orthosystem} to $F=G_{n,k}$, taking $\delta=1$, $r=\ell(n)$, and indices $d_{j}$ equal to the upper bounds of the parameter $s$ in the orthogonality conditions, we deduce that $G_{n,k}$ has at least $Z(n,k)-1$ zeros with odd multiplicity in $(a_{k}, b_{k})$.

If $\ell(n)\in[k:p-1]$, $\eta(n)+\ell(n)\in[k:p-1]$, and $\eta(n)>0$, then from \eqref{orthogredPsink} and \eqref{relbounds:6} we deduce that $G_{n,k}$ satisfies the following orthogonality conditions:
\begin{align*}
\int_{a_{k}}^{b_{k}}G_{n,k}(\tau)\tau^{s+1} d\mu_{k,j}(\tau) & =0,\qquad 0\leq s\leq \lambda-1,\quad j\in[k:\ell(n)],\\
\int_{a_{k}}^{b_{k}}G_{n,k}(\tau)\tau^{s} d\mu_{k,j}(\tau) & =0,\qquad 0\leq s\leq \lambda,\quad j\in[\ell(n)+1:\eta(n)+\ell(n)-1],\\
\int_{a_{k}}^{b_{k}}G_{n,k}(\tau)\tau^{s} d\mu_{k,j}(\tau) & =0,\qquad 0\leq s\leq \lambda-1,\quad j\in[\eta(n)+\ell(n):p-1],
\end{align*}
which implies by Lemma~\ref{lemma:orthosystem} that $G_{n,k}$ has at least $Z(n,k)-1$ zeros with odd multiplicity in $(a_{k}, b_{k})$.

If $\ell(n)\in[k:p-1]$ and $\eta(n)+\ell(n)\geq p+k$, then from \eqref{orthogredPsink} and \eqref{relbounds:7} we deduce:
\begin{align*}
\int_{a_{k}}^{b_{k}}G_{n,k}(\tau)\tau^{s+1} d\mu_{k,j}(\tau) & =0,\qquad 0\leq s\leq \lambda,\quad j\in[k:\eta(n)+\ell(n)-p-1],\\
\int_{a_{k}}^{b_{k}}G_{n,k}(\tau)\tau^{s+1} d\mu_{k,j}(\tau) & =0,\qquad 0\leq s\leq \lambda-1,\quad j\in[\eta(n)+\ell(n)-p:\ell(n)],\\
\int_{a_{k}}^{b_{k}}G_{n,k}(\tau)\tau^{s} d\mu_{k,j}(\tau) & =0,\qquad 0\leq s\leq \lambda,\quad j\in[\ell(n):p-1],
\end{align*}
showing again that $G_{n,k}=A\psi_{n,k}+B\psi_{n+1,k}$ has at least $Z(n,k)-1$ zeros with odd multiplicity in $(a_{k},b_{k})$.

To finish the proof of part 2), assume now that $\ell(n)\in[k:p-1]$ and $\eta(n)+\ell(n)\in[p:p+k-1]$. Then from \eqref{relbounds:8} we obtain
\begin{align*}
\int_{a_{k}}^{b_{k}}G_{n,k}(\tau)\tau^{s+1} d\mu_{k,j}(\tau) & =0,\qquad 0\leq s\leq \lambda-1,\quad j\in[k:\ell(n)],\\
\int_{a_{k}}^{b_{k}}G_{n,k}(\tau)\tau^{s} d\mu_{k,j}(\tau) & =0,\qquad 0\leq s\leq \lambda,\quad j\in[\ell(n)+1:p-1].
\end{align*}
This implies again that $G_{n,k}$ has at least $Z(n,k)-1$ zeros with odd multiplicity in $(a_{k},b_{k})$.

The proof of part 3) is left to the reader (apply \eqref{relbounds:9}--\eqref{relbounds:10}).
\end{proof}

As an immediate consequence of Lemmas~\ref{lemma:interlac:1} and \ref{lemma:interlac:2}, we obtain the following:

\begin{corollary}\label{cor:interl}
Assume that $A, B\in\mathbb{R}$, $|A|+|B|>0$, and let $k\in[0:p-1]$ and $n\geq 0$ be integers. Then all the zeros of $G_{n,k}$ in $\mathbb{C}\setminus([a_{k-1},b_{k-1}]\cup\{0\})$ are real and simple.
\end{corollary}

\subsection{Proof of Theorem~\ref{theo:interl}}

By definition, the zeros of $P_{n,k}$ are the zeros of $\psi_{n,k}$ in $(a_{k},b_{k})$, see Definition~\ref{Def:polyPnk}. By Proposition~\ref{prop:lpsinksummary}, these zeros are all simple.

First, let us show that the functions $\psi_{n,k}$ and $\psi_{n+1,k}$ cannot have a common zero in $(a_{k},b_{k})$. Assume the contrary, and let $x_{0}\in(a_{k},b_{k})$ satisfy $\psi_{n,k}(x_{0})=\psi_{n+1,k}(x_{0})=0$. Then we have $\psi_{n,k}'(x_{0})\neq 0$, $\psi_{n+1,k}'(x_{0})\neq 0$. Now take
\[
A=\begin{cases}
1, & \mathrm{if}\,\ell(n)\neq p,\\[0.5em]
1/x_{0}, & \mathrm{if}\,\ell(n)=p,
\end{cases}\qquad B=-\frac{\psi_{n,k}'(x_{0})}{\psi_{n+1,k}'(x_{0})},
\]
and consider the function $G_{n,k}$ given by \eqref{def:Gnk}. With this choice of $A$ and $B$ we obtain $G_{n,k}(x_{0})=G_{n,k}'(x_{0})=0$, which contradicts Corollary~\ref{cor:interl}.

Assume now that $\ell(n)\neq p$, and let $y\in(a_{k},b_{k})$ be arbitrary but fixed. Taking $A=\psi_{n+1,k}(y)$, $B=-\psi_{n,k}(y)$, we know by the argument in the previous paragraph that $|A|+|B|>0$. Since
\[
\psi_{n+1,k}(y)\,\psi_{n,k}(y)-\psi_{n,k}(y)\,\psi_{n+1,k}(y)=0,
\]
and the zeros on $(a_{k},b_{k})$ of $G_{n,k}(x)=\psi_{n+1,k}(y)\,\psi_{n,k}(x)-\psi_{n,k}(y)\,\psi_{n+1,k}(x)$ are simple, it follows that
\[
\psi_{n+1,k}(y)\,\psi_{n,k}'(y)-\psi_{n,k}(y)\,\psi_{n+1,k}'(y)\neq 0.
\]
But $\psi_{n+1,k}(y)\,\psi_{n,k}'(y)-\psi_{n,k}(y)\,\psi_{n+1,k}'(y)$ is a continuous real function on $(a_{k},b_{k})$, so it must have constant sign on this interval. Evaluating this function at two consecutive zeros of $\psi_{n+1,k}$, since the sign of $\psi_{n+1,k}'$ at these two points changes, the sign of $\psi_{n,k}$ must also change. By Bolzano's theorem we deduce that there must be an intermediate zero of $\psi_{n,k}$. Similarly, one proves that between two consecutive zeros of $\psi_{n,k}$ on $(a_{k},b_{k})$ there is one of $\psi_{n+1,k}$.

The argument in the case $\ell(n)=p$ is analogous, so we leave the analysis to the reader.

\section{Ratio asymptotics}\label{sec:ratioasymp}

\subsection{Main ideas}

Let us first outline the main ideas in the proof of ratio asymptotics for the polynomials $P_{n,k}$. The method we use for obtaining the ratio asymptotic results was first employed in \cite{AptLopRocha}. The argument goes as follows. Let $\rho\in\{0,\ldots,p(p+1)-1\}$ be fixed but arbitrary. We consider the $p$ families of ratios
\begin{equation}\label{eq:familyratios}
\left\{\frac{P_{\lambda p(p+1)+\rho+1,k}(z)}{P_{\lambda p(p+1)+\rho,k}(z)}\right\}_{\lambda\in\mathbb{N}},\qquad k=0,\ldots,p-1.
\end{equation}
For each $k$ fixed, the sequence \eqref{eq:familyratios} is uniformly bounded on compact subsets of $\mathbb{C}\setminus[a_{k},b_{k}]$, due to the interlacing property of the zeros of the polynomials $P_{n,k}$.

By Montel's theorem, there exists a subsequence $\Lambda\subset\mathbb{N}$ such that for each $k=0,\ldots,p-1$, the limit
\begin{equation}\label{eq:limratios}
\lim_{\lambda\in\Lambda}\frac{P_{\lambda p(p+1)+\rho+1,k}(z)}{P_{\lambda p(p+1)+\rho,k}(z)}=\widetilde{F}_{k}^{(\rho)}(z),\qquad z\in\mathbb{C}\setminus[a_{k},b_{k}]
\end{equation}
holds, uniformly on compact subsets of the indicated region. In principle, the limiting functions $\widetilde{F}_{k}^{(\rho)}(z)$ depend on the subsequence $\Lambda$, but it will be our main goal to show that in fact they are independent of $\Lambda$, proving this way the existence of the limits
\[
\lim_{\lambda\rightarrow\infty}\frac{P_{\lambda p(p+1)+\rho+1,k}(z)}{P_{\lambda p(p+1)+\rho,k}(z)}=\widetilde{F}_{k}^{(\rho)}(z),\qquad z\in\mathbb{C}\setminus[a_{k},b_{k}],
\]
for every $\rho$ and $k$ fixed as before.

In order to prove the independence of the functions $\widetilde{F}_{k}^{(\rho)}$ from $\Lambda$, we first identify these functions as Szeg\H{o} functions or Szeg\H{o} functions multiplied by certain conformal mappings, the Szeg\H{o} functions being associated with weights that can be expressed themselves in terms of the functions $\widetilde{F}^{(\rho)}_{k}$. This identification is accomplished using results on ratio and relative asymptotics of orthogonal polynomials with respect to varying measures that were obtained in \cite{BarCalLop}. Here we also apply the asymptotic formulas \eqref{converghnk}.

Using the boundary value properties of the Szeg\H{o} functions, we then show that a certain normalization $F_{k}^{(\rho)}$ of the functions $\widetilde{F}_{k}^{(\rho)}$ satisfies a system of boundary value problems. Then, to conclude the proof of the uniqueness of the limiting functions $\widetilde{F}_{k}^{(\rho)}$, it is enough to show that this boundary value problem has a unique solution.

\subsection{Asymptotics of the functions $h_{n,k}$}

A first step in the asymptotic analysis is to obtain the asymptotic behavior of the functions $h_{n,k}$. This is gathered in the following result.

\begin{proposition}
Assume that for each $k=0,\ldots,p-1$, the measure $\sigma_{k}^{*}$ has positive Radon-Nikodym derivative with respect to Lebesgue measure a.e. on $[a_{k},b_{k}]$. Then for all $k=0,\ldots,p-1$ and $\ell=0,\ldots,p,$ fixed,
\begin{equation}\label{convergtoequil}
p_{m(p+1)+\ell,k}^{2}(\tau)\,d|\nu_{m(p+1)+\ell,k}|(\tau)\xrightarrow[m\rightarrow\infty]{*} \frac{1}{\pi}\,\frac{d\tau}{\sqrt{(b_{k}-\tau)(\tau-a_{k})}}.
\end{equation}
Consequently, for each $k=1,\ldots,p$ and $\ell=0,\ldots,p$ fixed,
\begin{equation}\label{converghnk}
\lim_{m\rightarrow\infty}\varepsilon_{m(p+1)+\ell,k-1}\,h_{m(p+1)+\ell,k}(z)=\begin{cases}
\frac{z}{\sqrt{(z-b_{k-1})(z-a_{k-1})}} & \mathrm{if}\ \ell<k,\\[1em]
\frac{1}{\sqrt{(z-b_{k-1})(z-a_{k-1})}} & \mathrm{if}\ k\leq \ell,
\end{cases}
\end{equation}
uniformly on compact subsets of $\mathbb{C}\setminus[a_{k-1},b_{k-1}]$, where we take the branch of the square root such that $\sqrt{z}>0$ for $z$ real, $z>0$.
\end{proposition}
\begin{proof}
Taking $f(\tau)=\frac{1}{z-\tau}$ and using formula \eqref{intrep:hnk}, \eqref{converghnk} for $k+1$ follows directly from \eqref{convergtoequil} for $k$ and the well known identity
\[
\label{convghnk^*}
\frac{1}{\pi}\int_{a_{k}}^{b_{k}}\frac{1}{z-\tau}\,\frac{d\tau}{\sqrt{(b_{k}-\tau)(\tau-a_{k})}}=\frac{1}{\sqrt{(z-b_{k})(z-a_{k})}},\qquad z\in\mathbb{C}\setminus[a_{k},b_{k}].
\]
Therefore, we limit ourselves to proving \eqref{convergtoequil}. This is done by induction on $k$.

Take $n=m(p+1)+\ell$ with $\ell$ fixed, and so the measure $\sigma_{n,k}$ remains fixed as we let $m\rightarrow\infty$. In particular, for $k=0$, we have that
\[
d\sigma_{n,0}(\tau)=\begin{cases}
d\sigma_{0}^{*}(\tau), & \ell=0,\\
\tau\,d\sigma_{0}^{*}(\tau), & \ell>0,
\end{cases}
\qquad h_{n,0} \equiv 1,
\]
and according to \eqref{positivenunk},
\[
p_{m(p+1)+\ell,0}^{2}(\tau)\,d|\nu_{m(p+1)+\ell,0}|(\tau)=p_{m(p+1)+\ell,0}^{2}(\tau)\,\frac{d\sigma_{m(p+1)+\ell,0}(\tau)}{|P_{m(p+1)+\ell,1}
(\tau)|}.
 \]
Note that $\varepsilon_{m(p+1)+\ell,0}=1$. Since the zeros of the polynomials $P_{m(p+1)+\ell,1}$ are bounded away from $[a_{0},b_{0}]$ (the support of the measure $\sigma_{n,0}$) and $\deg(P_{m(p+1)+\ell,1})-2\deg(p_{m(p+1)+\ell,0})\leq 0$ (cf. \eqref{eq:ineqZnk}), it is straightforward to check that $(\{\sigma_{m(p+1)+\ell,0}\},\{P_{m(p+1)+\ell,1}\},l)$ is strongly admissible for every $l\in\mathbb{Z}$, in the sense of Definition 2 in \cite{BarCalLop}. As a consequence, by Corollary 3 in \cite{BarCalLop} we obtain that for every continuous function $f$ on $[a_{0},b_{0}]$,
\[
\lim_{m\rightarrow\infty}\int_{a_{0}}^{b_{0}}f(\tau)\,p_{m(p+1)+\ell,0}^{2}(\tau)\,\frac{d\sigma_{m(p+1)+\ell,0}(\tau)}{|P_{m(p+1)+\ell,1}
(\tau)|}=\frac{1}{\pi}\int_{a_{0}}^{b_{0}}\frac{f(\tau)\,d\tau}{\sqrt{(b_{0}-\tau)(\tau-a_{0})}},
\]
which is \eqref{convergtoequil} for $k=0$.

The basis of induction has been settled. Let us assume that \eqref{convergtoequil} holds for some $k-1, 0\leq k-1 \leq p-2$. We must prove that the same is true if $k-1$ is replaced by $k$.

From the definition we have
\[d|\nu_{m(p+1)+\ell,k}|(\tau) = \frac{|h_{m(p+1)+\ell,k}(\tau)| d|\sigma_{m(p+1)+\ell,k}|(\tau)}{|P_{m(p+1)+\ell,k-1}(\tau)P_{m(p+1)+\ell,k+1}(\tau)|}\]
where according to \eqref{varymeas:sigma}
\[
d|\sigma_{m(p+1)+\ell,k}|(\tau):=\begin{cases}
d\sigma_{k}^{*}(\tau), & \ell\leq k,\\
|\tau|\,d\sigma_{k}^{*}(\tau), & k<\ell.
\end{cases} \]
Since $\ell$ remains fixed this measure is one and the same for all $m$. On the other hand, as indicated in the first sentence of the proof, the induction hypothesis implies that \eqref{converghnk} takes place for $k$. In particular,
\[
\lim_{m\rightarrow\infty}|h_{m(p+1)+\ell,k}(\tau)|=\begin{cases}
\frac{|\tau|}{\sqrt{|(\tau-b_{k-1})(\tau-a_{k-1})|}} & \mathrm{if}\ \ell<k,\\[1em]
\frac{1}{\sqrt{|(\tau-b_{k-1})(\tau-a_{k-1})|}} & \mathrm{if}\ k\leq \ell,
\end{cases}
\]
uniformly on $[a_k,b_k]$.

Now, the zeros of the polynomials $P_{m(p+1)+\ell,k-1}P_{m(p+1)+\ell,k+1}$ are bounded away from $[a_{k},b_{k}]$ (the support of the measure $\sigma_{m(p+1)+\ell,k}$) and according to \eqref{asympZnk}
\[ \deg(P_{m(p+1)+\ell,k-1}) + \deg(P_{m(p+1)+\ell,k+1}) - 2 \deg(P_{m(p+1)+\ell,k-1}) = O(1), \qquad m \to \infty.
\]
Consequently,
$(\{|h_{m(p+1)+\ell,k}| |\sigma_{m(p+1)+\ell,k}|\},\{|P_{m(p+1)+\ell,k-1}P_{m(p+1)+\ell,k-1}|\},l)$ is strongly admissible for every $l\in\mathbb{Z}$, in the sense of Definition 2 in \cite{BarCalLop}. Therefore, by Corollary 3 in \cite{BarCalLop} we obtain \eqref{convergtoequil} for $k$ as needed.
\end{proof}

\begin{remark}
Since \eqref{convergtoequil} is valid for every fixed $\ell=0,\ldots,p,$ we have the weak limits
\[
p_{n,k}^{2}(\tau)\,d|\nu_{n,k}|(\tau)\xrightarrow[n\rightarrow\infty]{*} \frac{1}{\pi}\,\frac{d\tau}{\sqrt{(b_{k}-\tau)(\tau-a_{k})}}
\]
for each $k=0,\ldots,p-1$.
\end{remark}

\subsection{Preliminary analysis} \label{preliminary}

Throughout this section we assume that for all $k=0,\ldots,p-1$, the measure $\sigma_{k}^{*}$ has positive Radon-Nikodym derivative with respect to Lebesgue measure a.e. on $[a_{k},b_{k}]$. If $\{f_{n}\}_{n\in\widetilde{\Lambda}}$ is a sequence of analytic functions on an open domain $\Omega\subset\overline{\mathbb{C}}$, the notation
\[
\lim_{n\in\widetilde{\Lambda}} f_{n}(z)=F(z),\qquad z\in\Omega,
\]
will stand for the uniform convergence of $f_{n}$ to $F$ on each compact subset of $\Omega$. Recall that for a pair of integers $n\leq m$, the notation $[n:m]$ indicates the set of all integers $l$ satisfying $n\leq l\leq m$. Below we will continue using the notations $n=\lambda p(p+1)+\rho$ and \eqref{eq:decomp:rho}.

For a measurable function $f\geq 0$ defined on $[0,2\pi]$ such that $\log f\in L^{1}([0,2\pi],d\tau)$, let
\[
D(f;z):=\exp\left\{\frac{1}{4\pi}\int_{0}^{2\pi}\frac{e^{i\tau}+z}{e^{i\tau}-z}\,\log f(\tau)\,d\tau\right\},\qquad |z|\neq 1.
\]
If $w\geq 0$ is now a measurable function on an interval $[a,b]$ that satisfies the Szeg\H{o} condition
\[
\frac{\log w(t)}{\sqrt{(b-t)(t-a)}}\in L^{1}([a,b], dt),
\]
we denote with
\[
S(w;z):=\frac{1}{D(\widetilde{w};1/\Phi(z))}=D(\widetilde{w};\Phi(z)),
\]
the Szeg\H{o} function on $\overline{\mathbb{C}}\setminus[a,b]$ associated with $w$. In this formula, the function $\Phi=\Phi_{[a,b]}$ is the conformal mapping of $\overline{\mathbb{C}}\setminus[a,b]$ onto $\{|z|>1\}$ such that $\Phi(\infty)=\infty$ and $\Phi'(\infty)>0$, and $\widetilde{w}$ is the function on $[0,2\pi]$ given by
\[
\widetilde{w}(\theta)=w(l_{[a,b]}(\cos\theta)),
\]
where $l_{[a,b]}$ denotes the linear map that transforms $[-1,1]$ onto $[a,b]$, i.e.,
\[
l_{[a,b]}(x)=\frac{b-a}{2}\,x+\frac{b+a}{2}.
\]
A well-known property of the function $S(w;z)$ is that if $w$ is continuous at $x\in[a,b]$ and $w(x)>0$, then the limit
\begin{equation}\label{limbound:Szego}
\lim_{z\rightarrow x}|S(w;z)|^2=\frac{1}{w(x)}
\end{equation}
holds. This can be easily deduced for example from \cite[Theorem 1.2.4]{Ran}. With a slight abuse of notation, in this paper we will indicate \eqref{limbound:Szego} by writing $|S(w;x)|^{2}w(x)=1$ or an equivalent expression.

In the preliminary analysis that we will perform in this section, we fix $\rho$ in the expression $n=\lambda p(p+1)+\rho$, and let $\lambda$ tend to infinity along a certain subsequence. In particular, the quantities $\eta$ and $\ell$ in \eqref{eq:decomp:rho} will also remain fixed. The subsequence that we consider is a sequence $\Lambda\subset\mathbb{N}$ such that \eqref{eq:limratios} holds for each $k=0,\ldots,p-1$. Note that the functions $\widetilde{F}_{k}^{(\rho)}$ do not vanish in $\mathbb{C}\setminus [a_{k}, b_{k}]$. By convention, we define $\widetilde{F}_{-1}^{(\rho)} \equiv \widetilde{F}_{p}^{(\rho)}\equiv 1$.

The starting point of the analysis is the set of orthogonality conditions
\begin{align}
\int_{a_{k}}^{b_{k}} P_{n,k}(\tau)\,\tau^{s}\,d|\nu_{n,k}|(\tau) & =0,\qquad s=0,\ldots,Z(n,k)-1,\label{eq:orthogvaryPnk:1}\\
\int_{a_{k}}^{b_{k}} P_{n+1,k}(\tau)\,\tau^{s}\,d|\nu_{n+1,k}|(\tau) & =0,\qquad s=0,\ldots,Z(n+1,k)-1,\label{eq:orthogvaryPnk:2}
\end{align}
which follow from \eqref{orthog:lpnk}. Recall that the measures $|\nu_{n,k}|$, $|\nu_{n+1,k}|$ are given in \eqref{positivenunk}. Recall also that by convention $P_{n,-1} \equiv P_{n,p} \equiv 1$, see Definition \ref{Def:polyPnk}.

We subdivide the analysis into several cases, namely
\begin{itemize}
\item[Case 1)] $\ell(n)\in[0:k-2]$,
\item[Case 2)] $\ell(n)=k-1$,
\item[Case 3)] $\ell(n)=k$,
\item[Case 4)] $\ell(n)\in[k+1:p-1]$,
\item[Case 5)] $\ell(n)=p$.
\end{itemize}
Below we analyze these cases separately:

\smallskip

\noindent Case 1)  $\ell(n)\in[0:k-2]$. We have $\ell(n+1)=\ell(n)+1<k$; therefore, from \eqref{varymeas:sigma} we obtain that
\begin{equation}\label{relsigma}
d\sigma_{n,k}(\tau)=d\sigma_{n+1,k}(\tau)=d\sigma_{k}^{*}(\tau),
\end{equation}
and according to \eqref{positivenunk} we can write
\begin{equation}\label{eq:perturborthogmeas:1}
d|\nu_{n+1,k}|(\tau)=g_{n,k}(\tau)\,d|\nu_{n,k}|(\tau),
\end{equation}
where
\begin{equation}\label{def:gnk}
g_{n,k}(\tau):=\frac{|h_{n+1,k}(\tau)|}{|h_{n,k}(\tau)|}\frac{|P_{n,k-1}(\tau)|}{|P_{n+1,k-1}(\tau)|}\,\frac{|P_{n,k+1}(\tau)|}{|P_{n+1,k+1}(\tau)|}.
\end{equation}
Hence in \eqref{eq:perturborthogmeas:1} we have written the orthogonality measure in \eqref{eq:orthogvaryPnk:2} as a perturbation of the orthogonality measure in \eqref{eq:orthogvaryPnk:1}. If we now let $\lambda\rightarrow\infty$ along the sequence $\Lambda\subset\mathbb{N}$ and we keep $\rho$ fixed in $n=\lambda p(p+1)+\rho$, in virtue of \eqref{converghnk} and \eqref{eq:limratios} we have
\begin{equation}\label{eq:asympgnk:1}
\lim_{\lambda\in\Lambda} g_{n,k}(\tau)=\frac{1}{|\widetilde{F}_{k-1}^{(\rho)}(\tau)||\widetilde{F}_{k+1}^{(\rho)}(\tau)|},
\end{equation}
uniformly on $[a_{k},b_{k}]$.

Suppose now that $\eta(n)+\ell(n)\notin[k:p-1]$. Then, it follows from \eqref{eq:values:diffZnk:1} that $\deg(P_{n,k})=\deg(P_{n+1,k})$.  Applying Theorem 2 from \cite{BarCalLop} (result on relative asymptotics of polynomials orthogonal with respect to varying measures), from \eqref{eq:perturborthogmeas:1} and \eqref{eq:asympgnk:1} we deduce that
\begin{equation}\label{lim:Pnk:rel:1}
\lim_{\lambda\in\Lambda}\frac{P_{n+1,k}(z)}{P_{n,k}(z)}=\Fkr(z)=\frac{\Skr(z)}{\Skr(\infty)},\qquad z\in\overline{\mathbb{C}}\setminus[a_{k},b_{k}],
\end{equation}
where $S_{k}^{(\rho)}$ is the Szeg\H{o} function on $\overline{\mathbb{C}}\setminus[a_{k},b_{k}]$ associated with the weight $(|\Fkmr(\tau)||\Fkpr(\tau)|)^{-1}$, $\tau\in[a_{k},b_{k}]$. Due to \eqref{limbound:Szego} and \eqref{lim:Pnk:rel:1}, taking limit as $z\to \tau, \tau \in [a_k,b_k]$, we obtain
\begin{equation}\label{ecuac1}
 \frac{|\Fkr(\tau)|^2}{|\widetilde{F}_{k-1}^{(\rho)}(\tau)||\widetilde{F}_{k+1}^{(\rho)}(\tau)|}=
\frac{|\Skr(\tau)|^2}{|\Skr(\infty)|^2|\widetilde{F}_{k-1}^{(\rho)}(\tau)||\widetilde{F}_{k+1}^{(\rho)}(\tau)|} = \frac{1}{w_k^{(\rho)}},\qquad \tau\in [a_{k},b_{k}],
\end{equation}
where $w_k^{(\rho)} = |\Skr(\infty)|^2 > 0$.

If we assume that $\eta(n)+\ell(n)\in[k:p-1]$, then we have $\deg(P_{n+1,k})=\deg(P_{n,k})+1$, cf. \eqref{eq:values:diffZnk:1}. In order to analyze the ratio $P_{n+1,k}/P_{n,k}$ in this case,  we introduce an auxiliary polynomial $P_{n,k}^{*}$, defined as the monic polynomial of degree $\deg(P_{n+1,k})=\deg(P_{n,k})+1$ that is orthogonal with respect to the measure $d|\nu_{n,k}|(\tau)$.

Then, by Theorem 1 from \cite{BarCalLop} (result on ratio asymptotics of polynomials orthogonal with respect to varying measures), we obtain that
\[
\lim_{\lambda\in\Lambda}\frac{P_{n,k}^{*}(z)}{P_{n,k}(z)}=\frac{\phi_{k}(z)}{\phi_{k}'(\infty)},\qquad z\in\mathbb{C}\setminus[a_{k},b_{k}],
\]
where $\phi_{k}$ denotes the conformal mapping from $\overline{\mathbb{C}}\setminus[a_{k},b_{k}]$ onto the exterior of the unit circle which satisfies $\phi_{k}(\infty)=\infty$ and $\phi_{k}'(\infty)>0$. Now $P_{n,k}^{*}$ and $P_{n+1,k}$ have the same degree, the first polynomial is orthogonal with respect to $|\nu_{n,k}|$ and the second one is orthogonal with respect to $g_{n,k}\, d|\nu_{n,k}|$, so by the relative asymptotic result mentioned above we have
\[
\lim_{\lambda\in\Lambda}\frac{P_{n+1,k}(z)}{P_{n,k}^{*}(z)}=\frac{S_{k}^{(\rho)}(z)}{S_{k}^{(\rho)}(\infty)},\qquad z\in\overline{\mathbb{C}}\setminus[a_{k},b_{k}],
\]
where $S_{k}^{(\rho)}$ is again the Szeg\H{o} function on $\overline{\mathbb{C}}\setminus[a_{k},b_{k}]$ associated with the weight $(|\widetilde{F}_{k-1}^{(\rho)}(\tau)||\widetilde{F}_{k+1}^{(\rho)}(\tau)|)^{-1}$, $\tau\in[a_{k},b_{k}]$. So it follows that if $\eta(n)+\ell(n)\in[k:p-1]$ then
\begin{equation}\label{lim:Pnk:ratio:rel:1}
\lim_{\lambda\in\Lambda}\frac{P_{n+1,k}(z)}{P_{n,k}(z)}
=\widetilde{F}_{k}^{(\rho)}(z)=\frac{S_{k}^{(\rho)}(z)}{S_{k}^{(\rho)}(\infty)}\,\frac{\phi_{k}(z)}{\phi_{k}'(\infty)},\qquad z\in\mathbb{C}\setminus[a_{k},b_{k}].
\end{equation}
Taking account of \eqref{limbound:Szego} and \eqref{lim:Pnk:ratio:rel:1}, we conclude that
\begin{equation}\label{ecuac2}
 \frac{|\Fkr(\tau)|^2}{|\widetilde{F}_{k-1}^{(\rho)}(\tau)||\widetilde{F}_{k+1}^{(\rho)}(\tau)|}=
\frac{|\Skr(\tau)|^2}{|\Skr(\infty)\phi_k^\prime(\infty)|^2|\widetilde{F}_{k-1}^{(\rho)}(\tau)||\widetilde{F}_{k+1}^{(\rho)}(\tau)|} = \frac{1}{w_k^{(\rho)}},\qquad \tau\in [a_{k},b_{k}],
\end{equation}
with $w_k^{(\rho)} = |\Skr(\infty)\phi_k^\prime(\infty)|^2 > 0$.
This finishes the analysis of Case 1). Since Cases 4) and 1) are rather similar, we now analyze

\smallskip

\noindent Case 4)  $\ell(n)\in[k+1:p-1]$. We have $\ell(n+1)=\ell(n)+1$ and so $\ell(n+1)>k$; therefore, we deduce from \eqref{varymeas:sigma} that
\[
d\sigma_{n,k}(\tau)=d\sigma_{n+1,k}(\tau)=\tau d\sigma_{k}^{*}(\tau),
\]
and this implies that \eqref{eq:perturborthogmeas:1}--\eqref{def:gnk} hold. Applying \eqref{eq:limratios} and \eqref{converghnk} we also obtain \eqref{eq:asympgnk:1}, uniformly on $[a_{k},b_{k}]$.

Assume that $\eta(n)+\ell(n)\notin[p:p+k-1]$. Then, it follows from \eqref{eq:values:diffZnk:1} that $\deg(P_{n,k})=\deg(P_{n+1,k})$. Applying Theorem 2 from \cite{BarCalLop} as in Case 1), from \eqref{eq:perturborthogmeas:1} and \eqref{eq:asympgnk:1} we obtain
\begin{equation}\label{lim:Pnk:rel:2}
\lim_{\lambda\in\Lambda}\frac{P_{n+1,k}(z)}{P_{n,k}(z)}=\widetilde{F}_{k}^{(\rho)}(z)
=\frac{S_{k}^{(\rho)}(z)}{S_{k}^{(\rho)}(\infty)},\qquad z\in\overline{\mathbb{C}}\setminus[a_{k},b_{k}],
\end{equation}
where $S_{k}^{(\rho)}$ is the Szeg\H{o} function on $\overline{\mathbb{C}}\setminus[a_{k},b_{k}]$ associated with the weight $(|\widetilde{F}_{k-1}^{(\rho)}(\tau)||\widetilde{F}_{k+1}^{(\rho)}(\tau)|)^{-1}$, $\tau\in[a_{k},b_{k}]$. In this situation, \eqref{limbound:Szego} and \eqref{lim:Pnk:rel:2} imply
\begin{equation}\label{ecuac3}
 \frac{|\Fkr(\tau)|^2}{|\widetilde{F}_{k-1}^{(\rho)}(\tau)||\widetilde{F}_{k+1}^{(\rho)}(\tau)|}=
\frac{|\Skr(\tau)|^2}{|\Skr(\infty)|^2|\widetilde{F}_{k-1}^{(\rho)}(\tau)||\widetilde{F}_{k+1}^{(\rho)}(\tau)|} = \frac{1}{w_k^{(\rho)}},\qquad \tau\in [a_{k},b_{k}],
\end{equation}
where $w_k^{(\rho)} = |\Skr(\infty)|^2 > 0$ (the same as in \eqref{ecuac1}).

If $\eta(n)+\ell(n)\in[p:p+k-1]$, then according to \eqref{eq:values:diffZnk:1} we have $\deg(P_{n+1,k})=\deg(P_{n,k})-1$, i.e., $P_{n+1,k}$ is of degree one unit less than the degree of $P_{n,k}$. Arguing as in Case 1) with the help of an auxiliary polynomial $P_{n,k}^{*}$ (of degree $\deg(P_{n+1,k})=\deg(P_{n,k})-1$ and orthogonal with respect to $|\nu_{n,k}|$), we obtain that in this case
\begin{equation}\label{lim:Pnk:ratio:rel:2}
\lim_{\lambda\in\Lambda}\frac{P_{n+1,k}(z)}{P_{n,k}(z)}
=\widetilde{F}_{k}^{(\rho)}(z)=\frac{S_{k}^{(\rho)}(z)}{S_{k}^{(\rho)}(\infty)}\,\frac{\phi_{k}'(\infty)}{\phi_{k}(z)},\qquad z\in\overline{\mathbb{C}}\setminus[a_{k},b_{k}],
\end{equation}
where $\phi_{k}$ is exactly as before and $\Skr$ is again the  Szeg\H{o} function associated with the weight $(|\widetilde{F}_{k-1}^{(\rho)}(\tau)||\widetilde{F}_{k+1}^{(\rho)}(\tau)|)^{-1}$, $\tau\in[a_{k},b_{k}]$.
Now, \eqref{limbound:Szego} and \eqref{lim:Pnk:ratio:rel:2} imply
\begin{equation}\label{ecuac4}
 \frac{|\Fkr(\tau)|^2}{|\widetilde{F}_{k-1}^{(\rho)}(\tau)||\widetilde{F}_{k+1}^{(\rho)}(\tau)|}=
\frac{|\Skr(\tau)\phi_k^{\prime}(\infty)|^2}{|\Skr(\infty)|^2|\widetilde{F}_{k-1}^{(\rho)}(\tau)||\widetilde{F}_{k+1}^{(\rho)}(\tau)|} = \frac{1}{w_k^{(\rho)}},\qquad \tau\in [a_{k},b_{k}],
\end{equation}
where $w_k^{(\rho)} = |\Skr(\infty)/\phi_k^{\prime}(\infty)|^2 > 0$.
This concludes the analysis of Case $4)$.

\smallskip

\noindent Case 2) $\ell(n)=k-1$. Here, $\ell(n+1)=\ell(n)+1=k$, hence \eqref{relsigma} and \eqref{eq:perturborthogmeas:1}--\eqref{def:gnk} hold. However, in this case it is convenient to write
\[
g_{n,k}(\tau)=\frac{1}{|\tau|}\,\widetilde{g}_{n,k}(\tau),\qquad \widetilde{g}_{n,k}(\tau):=|\tau|\frac{|h_{n+1,k}(\tau)|}{|h_{n,k}(\tau)|}\frac{|P_{n,k-1}(\tau)|}{|P_{n+1,k-1}(\tau)|}\,\frac{|P_{n,k+1}(\tau)|}{|P_{n+1,k+1}(\tau)|}.
\]
From \eqref{intrep:hnk}, \eqref{eq:limratios}, and \eqref{converghnk} we obtain
\begin{equation}
\label{eq:asympgnk:3}
\lim_{\lambda\in\Lambda}\widetilde{g}_{n,k}(\tau)=\frac{1}{|\Fkmr(\tau)||\Fkpr(\tau)|},
\end{equation}
uniformly on $[a_{k},b_{k}]$.

In view of this asymptotic behavior of $\widetilde{g}_{n,k}$, using the auxiliary asymptotic results from \cite{BarCalLop} we conclude that if $\eta(n)+\ell(n)\notin[k:p-1]$, then
\begin{equation}\label{lim:Pnk:rel:3}
\lim_{\lambda\in\Lambda}\frac{P_{n+1,k}(z)}{P_{n,k}(z)}=\widetilde{F}_{k}^{(\rho)}(z)
=\frac{S_{k}^{(\rho)}(z)}{S_{k}^{(\rho)}(\infty)},\qquad z\in\overline{\mathbb{C}}\setminus[a_{k},b_{k}],
\end{equation}
where $S_{k}^{(\rho)}$ is the Szeg\H{o} function on $\overline{\mathbb{C}}\setminus[a_{k},b_{k}]$ associated with the weight $(|\tau||\widetilde{F}_{k-1}^{(\rho)}(\tau)||\widetilde{F}_{k+1}^{(\rho)}(\tau)|)^{-1}$, $\tau\in[a_{k},b_{k}]$, and if $\eta(n)+\ell(n)\in[k:p-1]$, then
\begin{equation}\label{lim:Pnk:ratio:rel:3}
\lim_{\lambda\in\Lambda}\frac{P_{n+1,k}(z)}{P_{n,k}(z)}
=\widetilde{F}_{k}^{(\rho)}(z)=\frac{S_{k}^{(\rho)}(z)}{S_{k}^{(\rho)}(\infty)}\,\frac{\phi_{k}(z)}{\phi_{k}'(\infty)},\qquad z\in\overline{\mathbb{C}}\setminus[a_{k},b_{k}],
\end{equation}
with the same definition of $\Skr$ as in \eqref{lim:Pnk:rel:3}. From \eqref{limbound:Szego} and \eqref{lim:Pnk:rel:3}--\eqref{lim:Pnk:ratio:rel:3}, we conclude that
\begin{equation}\label{ecuac5}
 \frac{|\Fkr(\tau)|^2}{|\tau||\widetilde{F}_{k-1}^{(\rho)}(\tau)||\widetilde{F}_{k+1}^{(\rho)}(\tau)|}=\frac{1}{w_k^{(\rho)}},\qquad \tau\in [a_{k},b_{k}] \setminus \{0\},
\end{equation}
where
\[
w_k^{(\rho)}=\begin{cases}|\Skr(\infty)|^2>0, & \ell(n)=k-1,\quad \eta(n)+\ell(n)\notin[k:p-1],\\[0.5em]
|\Skr(\infty)\phi_k^{\prime}(\infty)|^2>0, & \ell(n)=k-1,\quad \eta(n)+\ell(n)\in[k:p-1].
\end{cases}
\]

\smallskip

\noindent Case 3) $\ell(n)=k$. Since $0\leq k\leq p-1$, now we have $\ell(n+1)=\ell(n)+1>k$, and so
\begin{align*}
d\sigma_{n,k}(\tau) & =d\sigma_{k}^{*}(\tau),\\
d\sigma_{n+1,k}(\tau) & =\tau\,d\sigma_{k}^{*}(\tau).
\end{align*}
Then,
\[
d|\nu_{n+1,k}|(\tau)=g_{n,k}(\tau)\,|\tau|\,d|\nu_{n,k}|(\tau),
\]
with $g_{n,k}$ as in \eqref{def:gnk}, and we also have
\[
\lim_{\lambda\in\Lambda} g_{n,k}(\tau)\,|\tau|=\frac{|\tau|}{|\Fkmr(\tau)||\Fkpr(\tau)|},
\]
uniformly on $[a_{k},b_{k}]$.

If $\eta(n)+\ell(n)\notin[p:p+k-1]$, then $\deg(P_{n+1,k})=\deg(P_{n,k})$, and as in Case 1), we obtain
\begin{equation}\label{lim:Pnk:rel:4}
\lim_{\lambda\in\Lambda}\frac{P_{n+1,k}(z)}{P_{n,k}(z)}
=\widetilde{F}_{k}^{(\rho)}(z)=\frac{S_{k}^{(\rho)}(z)}{S_{k}^{(\rho)}(\infty)},\qquad z\in\overline{\mathbb{C}}\setminus[a_{k},b_{k}],
\end{equation}
where $S_{k}^{(\rho)}$ is the Szeg\H{o} function on $\overline{\mathbb{C}}\setminus[a_{k},b_{k}]$ associated with the weight $|\tau|(|\widetilde{F}_{k-1}^{(\rho)}(\tau)||\widetilde{F}_{k+1}^{(\rho)}(\tau)|)^{-1}$, $\tau\in[a_{k},b_{k}]$.

If $\eta(n)+\ell(n)\in [p:p+k-1]$ then $\deg(P_{n,k+1})=\deg(P_{n,k})-1$, and as in Case 4) we obtain the asymptotic formula
\begin{equation}\label{lim:Pnk:ratio:rel:4}
\lim_{\lambda\in\Lambda}\frac{P_{n+1,k}(z)}{P_{n,k}(z)}
=\widetilde{F}_{k}^{(\rho)}(z)=\frac{S_{k}^{(\rho)}(z)}{S_{k}^{(\rho)}(\infty)}\,\frac{\phi_{k}'(\infty)}{\phi_{k}(z)},\qquad z\in\overline{\mathbb{C}}\setminus[a_{k},b_{k}],
\end{equation}
where $\Skr$ is defined as in \eqref{lim:Pnk:rel:4}. From \eqref{limbound:Szego}, \eqref{lim:Pnk:rel:4}, and \eqref{lim:Pnk:ratio:rel:4}, if follows that
\begin{equation}\label{ecuac6}
 \frac{|\tau||\Fkr(\tau)|^2}{|\widetilde{F}_{k-1}^{(\rho)}(\tau)||\widetilde{F}_{k+1}^{(\rho)}(\tau)|}  = \frac{1}{w_k^{(\rho)}},\qquad \tau\in [a_{k},b_{k}] \setminus \{0\},
\end{equation}
where
\[
w_k^{(\rho)} =\begin{cases}
|\Skr(\infty)|^2 > 0, & \ell(n) = k,\quad \eta(n)+\ell(n) \notin [p:p+k-1], \\[0.5em]
|\Skr(\infty)/\phi_k^{\prime}(\infty)|^2 > 0, & \ell(n) = k, \quad \eta(n)+\ell(n) \in [p:p+k-1].
\end{cases}
\]

\smallskip

\noindent Case 5)  $\ell(n)=p$. In particular we have $\ell>k$ and $k\geq \ell(n+1)=0$. Hence in this case
\begin{align*}
d\sigma_{n,k}(\tau) & =\tau\,d\sigma_{k}^{*}(\tau),\\
d\sigma_{n+1,k}(\tau) & =d\sigma_{k}^{*}(\tau),
\end{align*}
and so
\[
d|\nu_{n+1,k}|(\tau)=\frac{g_{n,k}(\tau)}{|\tau|}\,d|\nu_{n,k}|(\tau),
\]
where $g_{n,k}$ is again given by \eqref{def:gnk}. In virtue of \eqref{converghnk} we have
\begin{equation}\label{lim:hnk:new}
\lim_{\lambda\in\Lambda}\frac{|h_{n+1,k}(\tau)|}{|h_{n,k}(\tau)|}=\begin{cases}
|\tau| & \mbox{if}\,\,k\geq 1,\\[0.5em]
1 & \mbox{if}\,\,k=0,
\end{cases}
\end{equation}
uniformly on $[a_{k},b_{k}]$. According to \eqref{eq:values:diffZnk:2} we have
\[
\deg(P_{n+1,k})-\deg(P_{n,k})=\begin{cases}
0 & \mbox{if}\,\,\eta(n)\in[0:k-1],\\[0.5em]
1 & \mbox{if}\,\,\eta(n)\in [k:p-1].
\end{cases}
\]
If we assume that $\eta(n)\in[0:k-1]$, then in particular $k\geq 1$ and from \eqref{eq:limratios} and \eqref{lim:hnk:new} we get
\[
\lim_{\lambda\in\Lambda}\frac{g_{n,k}(\tau)}{|\tau|}=\frac{1}{|\widetilde{F}_{k-1}^{(\rho)}(\tau)||\widetilde{F}_{k+1}^{(\rho)}(\tau)|},\qquad \tau\in[a_{k},b_{k}],
\]
and since $\deg(P_{n+1,k})=\deg(P_{n,k})$, we deduce that the asymptotic formula \eqref{lim:Pnk:rel:1} holds with $S_{k}^{(\rho)}$ defined in the same way as in that formula.

Suppose now that $\eta(n)\in[k:p-1]$. Then we have
\begin{equation}\label{eq:asympgnk:2}
\lim_{\lambda\in\Lambda}\frac{g_{n,k}(\tau)}{|\tau|}=\begin{cases}\frac{1}{|\tau||\widetilde{F}_{1}^{(\rho)}(\tau)|} & \mbox{if}\,\,k=0,\\[0.5em]
\frac{1}{|\widetilde{F}_{k-1}^{(\rho)}(\tau)||\widetilde{F}_{k+1}^{(\rho)}(\tau)|} & \mbox{if}\,\,k\geq 1,
\end{cases}\qquad \tau\in[a_{k},b_{k}],
\end{equation}
and $\deg(P_{n+1,k})=1+\deg(P_{n,k})$. Arguing as before we obtain that the asymptotic formula \eqref{lim:Pnk:ratio:rel:1} is valid with $S_{k}^{(\rho)}(z)$ being now the Szeg\H{o} function associated with the weight indicated on the right-hand side of \eqref{eq:asympgnk:2}. On account of \eqref{limbound:Szego} and the structure of $\widetilde{F}_k^{(\rho)}$ for the different values of $k$, now we have
\begin{equation}
\label{ecuac7}
\begin{array}{ll}
\displaystyle{\frac{|\widetilde{F}_{0}^{(\rho)}(\tau)|^{2}}{|\tau||\widetilde{F}_{1}^{(\rho)}(\tau)|} =\frac{1}{w_0^{(\rho)}}}, & \tau\in [a_{0},b_{0}]  \setminus \{0\},\\
\displaystyle{\frac{|\widetilde{F}_{k}^{(\rho)}(\tau)|^{2}}{|\widetilde{F}_{k-1}^{(\rho)}(\tau)||\widetilde{F}_{k+1}^{(\rho)}(\tau)|} =\frac{1}{w_k^{(\rho)}}},&  \tau\in[a_{k},b_{k}], \,\,\,\, k\in [1:p-1],
\end{array}
\end{equation}
where $w_k^{(\rho)} = |S_k^{(\rho)}(\infty)|^2, \eta(n) + 1 \leq k \leq p-1$ and $w_k^{(\rho)} = |S_k^{(\rho)}(\infty)\phi_k^{\prime}(\infty)|^2, 0 \leq k \leq \eta(n)$.

\subsection{Boundary value problem for the functions $\Fkr$}

The preliminary analysis carried out in the previous section leads to the following boundary value relations between the functions $\Fkr$.

\begin{lemma}\label{boundary}
Let $\rho\in[0:p(p+1)-1]$ be fixed, and let $\ell\in[0:p]$ be the remainder in the division of $\rho$ by $p+1$. Let $\Fkr$, $0\leq k\leq p-1$, be any collection of functions obtained through the asymptotic formula \eqref{eq:limratios} for some subsequence $\Lambda\subset\mathbb{N}$. Then there exist positive constants $c_{k}^{(\rho)}$ so that the collection of functions $F_{k}^{(\rho)}(z)=c_{k}^{(\rho)} \Fkr(z), 0\leq k \leq p-1$ satisfies the following systems of boundary value equations:
\begin{itemize}
\item[1)] When $\ell \in [0:p-1]$ (here, $[0:-1], [p:p-1]$ and $[p+1:p-1]$ denote the empty set for the corresponding values of $\ell$)
\[
\begin{array}{lll}
\displaystyle{\frac{|\Fk(\tau)|^{2}}{|\Fkm(\tau)||\Fkp(\tau)|}=1}, & \tau\in[a_{k},b_{k}], & k \in [0:\ell -1] \cup [\ell +2:p-1], \\
\displaystyle{\frac{|\Fk(\tau)|^{2}\,|\tau|}{|\Fkm(\tau)||\Fkp(\tau)|}=1}, & \tau\in[a_{k},b_{k}] \setminus \{0\}, & k = \ell, \\
\displaystyle{\frac{|\Fk(\tau)|^{2}}{|\tau||\Fkm(\tau)||\Fkp(\tau)|}=1}, & \tau\in[a_{k},b_{k}] \setminus \{0\}, & k = \ell +1.
\end{array}
\]
(The last equation is dropped if $\ell = p-1$.)
\item[2)] For $\ell=p$, the system is
\[
\begin{array}{ll}
\displaystyle{\frac{|F_{0}^{(\rho)}(\tau)|^{2}}{|\tau||F_{1}^{(\rho)}(\tau)|} =1}, & \tau\in [a_{0},b_{0}] \setminus \{0\},\\
\displaystyle{\frac{|\Fk(\tau)|^{2}}{|\Fkm(\tau)||\Fkp(\tau)|} =1},&  \tau\in[a_{k},b_{k}],\quad k\in [1:p-1].
\end{array}
\]
\end{itemize}
Moreover, for each $\rho$ fixed, the functions $F_{k}^{(\rho)}(z), 0\leq k \leq p-1$ satisfy:
\begin{itemize}
\item[i)] $(F_k^{(\rho)})^{\pm 1} \in \mathcal{H}(\mathbb{C}\setminus [a_k,b_k])$.
\item[ii)] The leading coefficient (corresponding to the highest power of $z$) of the Laurent expansion of $F_k^{(\rho)}$  at $\infty$ is positive.
\item[iii)] $F_k^{(\rho)}$ either has a simple pole, a simple zero, or takes a finite positive value at $\infty$. For a given $\rho \in [0:p(p+1)-1]$ and $k \in [0:p-1]$, only one of these situations occur independently of $\Lambda$.
 \end{itemize}
\end{lemma}
\begin{proof} Let $\Fkr$, $0\leq k\leq p-1$, be any collection of functions obtained through the asymptotic formula \eqref{eq:limratios} for some subsequence $\Lambda\subset\mathbb{N}$.
From \eqref{ecuac1}, \eqref{ecuac2}, \eqref{ecuac3}, \eqref{ecuac4}, \eqref{ecuac5}, \eqref{ecuac6}, and \eqref{ecuac7}, we obtain the following systems of boundary value problems for each fixed $\rho$ making $k$ range from $0$ to $p-1$.
\begin{itemize}
\item[1)] If $\ell \in [0:p-1]$,
\[
\begin{array}{lll}
\displaystyle{\frac{|\widetilde{F}_{k}^{(\rho)}(\tau)|^{2}}{|\widetilde{F}_{k-1}^{(\rho)}(\tau)||\widetilde{F}_{k+1}^{(\rho)}(\tau)|}=\frac{1}{w_k^{(\rho)}}}, & \tau\in[a_{k},b_{k}], & k \in [0:\ell -1] \cup [\ell +2:p-1], \\
\displaystyle{\frac{|\widetilde{F}_{k}^{(\rho)}(\tau)|^{2}\,|\tau|}{|\widetilde{F}_{k-1}^{(\rho)}(\tau)||\widetilde{F}_{k+1}^{(\rho)}(\tau)|}=\frac{1}{w_k^{(\rho)}}}, & \tau\in[a_{k},b_{k}] \setminus \{0\}, & k = \ell, \\
\displaystyle{\frac{|\widetilde{F}_{k}^{(\rho)}(\tau)|^{2}}{|\tau||\widetilde{F}_{k-1}^{(\rho)}(\tau)||\widetilde{F}_{k+1}^{(\rho)}(\tau)|}=\frac{1}{w_k^{(\rho)}}}, & \tau\in[a_{k},b_{k}] \setminus \{0\}, & k = \ell +1.
\end{array}
\]
\item[2)] For $\ell=p$, the system is
\[
\begin{array}{ll}
\displaystyle{\frac{|\widetilde{F}_{0}^{(\rho)}(\tau)|^{2}}{|\tau||\widetilde{F}_{1}^{(\rho)}(\tau)|} =\frac{1}{w_0^{(\rho)}}}, & \tau\in [a_{0},b_{0}]  \setminus \{0\},\\
\displaystyle{\frac{|\widetilde{F}_{k}^{(\rho)}(\tau)|^{2}}{|\widetilde{F}_{k-1}^{(\rho)}(\tau)||\widetilde{F}_{k+1}^{(\rho)}(\tau)|} =\frac{1}{w_k^{(\rho)}}},&  \tau\in[a_{k},b_{k}],\,\,\,\,  k\in [1:p-1] .
\end{array}
\]
\end{itemize}
The values $w_k^{(\rho)}, 0\leq k\leq p-1$, which depend on $\rho$ (i.e., on $\ell$ and $\eta$) were specified in the preliminary analysis. They are all positive.

Set $F_{k}^{(\rho)}(z)=c_{k}^{(\rho)} \Fkr(z), c_k^{(\rho)} > 0$ and substitute in the previous systems. Our problem reduces to determining if there exist positive constants $c_{k}^{(\rho)}, k=0,\ldots,p-1\,\, (c_{-1}^{(\rho)} = c_{p}^{(\rho)} =1 )$ such that
\[\frac{(c_{k}^{(\rho)})^2}{c_{k-1}^{(\rho)}c_{k+1}^{(\rho)}w_k^{(\rho)}} = 1, \qquad k=0,\ldots,p-1.\]
Taking logarithm, this is equivalent to determining if the nonhomogeneous linear system of equations on $\log c_k^{(\rho)}$ is determinate
\begin{equation}
\label{eq:cs}
 2 \log c_k^{(\rho)} - \log c_{k-1}^{(\rho)} - \log c_{k+1}^{(\rho)} = \log w_k^{(\rho)}, \qquad k=0,\ldots,p-1.
\end{equation}
Of course this is the case.

Since ${F}_k^{(\rho)} = c_k^{(\rho)}\widetilde{F}_k^{(\rho)}, c_k^{(\rho)} > 0$, the properties i)-iii) are inherited from analogous ones for  $\widetilde{F}_k^{(\rho)}, 0\leq k\leq p-1$. Therefore, it is sufficient to verify them for the $\widetilde{F}_k^{(\rho)}$. First, $\widetilde{F}_k^{(\rho)}  \in \mathcal{H}(\mathbb{C}\setminus [a_k,b_k])$ because it is the uniform limit on compact subsets of $\mathbb{C}\setminus [a_k,b_k]$ of holomorphic functions. Since the zeros of these  functions all lie in $[a_k,b_k]$, by  Hurwitz' theorem $\widetilde{F}_k^{(\rho)}$ has no zero in $\mathbb{C}\setminus [a_k,b_k]$; therefore, its reciprocal is also in $\mathcal{H}(\mathbb{C}\setminus [a_k,b_k])$. Regarding ii) and iii),  notice that from the preliminary analysis it follows that $\widetilde{F}_k^{(\rho)}$ is a normalized Szeg\H{o} function in the complement of $[a_k,b_k]$, or a normalized Szeg\H{o} function multiplied or divided by $\phi_k$ which has a simple pole at $\infty$. In each case the normalization is taken so that the leading coefficient of the Laurent expansion of $\widetilde{F}_k^{(\rho)}$ at $\infty$ is equal to $1$. Finally, as we have seen, the existence of a pole, a zero, or a finite value of $\widetilde{F}_k^{(\rho)}$ at $\infty$ only depends on $\ell(n)$ and $\eta(n)$ in the decomposition \eqref{eq:decomp:rho} of $\rho$. Since $\rho$ is fixed so are $\ell(n)$ and $\eta(n)$ on the set of indices $\Lambda_\rho = \{n: n = \lambda p(p+1) + \rho, \lambda \in \mathbb{Z}_+ \}.$ We have concluded the proof.
\end{proof}

\begin{lemma}\label{unique}
Let $0\leq \rho\leq p(p+1)-1$ be fixed, and let $\widetilde{F}_k^{(\rho)}$, $0\leq k\leq p-1$ be any collection of functions obtained through the asymptotic formula \eqref{eq:limratios} for some subsequence $\Lambda\subset\mathbb{N}$. Let $F_{k}^{(\rho)}=c_{k}^{(\rho)} \widetilde{F}_{k}^{(\rho)}$, $0\leq k\leq p-1$, be the corresponding collection of functions obtained as in Lemma \ref{boundary}. Then, the collection $F_k^{(\rho)}$, $0\leq k\leq p-1$ is uniquely determined.
\end{lemma}

\begin{proof} Let $F_{k}^{(\rho)}$, $k=0,\ldots,p-1$ and $G_{k}^{(\rho)}$, $k=0,\ldots,p-1,$ be two collections of functions satisfying the hypotheses of the Lemma. Then, by Lemma~\ref{boundary} we know that both collections satisfy properties i)-iii) (with a pole, a zero, or a finite value at $\infty$ for the same values of $k$) and the same system of boundary value equations. Let us construct a third collection of functions $H_k^{(\rho)}=F_k^{(\rho)}/G_k^{(\rho)}, k=0,\ldots,p-1$. From i)-iii) it follows that   $H_k^{(\rho)} \in \mathcal{H}(\overline{\mathbb{C}} \setminus [a_k,b_k]),  k=0,\ldots,p-1$ and the leading coefficients of these functions are positive. The collection $H_k^{(\rho)} , k=0,\ldots,p-1$ verifies the system of boundary value equations
\begin{equation}
\label{boundaryH}
\frac{|H_{k}^{(\rho)}(\tau)|^{2}}{|H_{k-1}^{(\rho)}(\tau)||H_{k+1}^{(\rho)}(\tau)|}=1 , \qquad \tau\in[a_{k},b_{k}], \qquad k = 0,\ldots,p-1,
\end{equation}
where $H_{-1}^{(\rho)} = H_p^{(\rho)} \equiv 1$. Here, we have included the point $0$ in the boundary value condition, even if it is an extreme point of $[a_k,b_k]$, because $H_{k}^{(\rho)}$ is, up to a constant factor, the Szeg\H{o} function of $|H_{k-1}^{(\rho)}H_{k+1}^{(\rho)}|^{-1}$ which is continuous and different from zero on all $[a_k,b_k]$ (recall that the Szeg\H{o} function is multiplicative).

Taking logarithm in \eqref{boundaryH}, we obtain the functional homogeneous system of equations
\begin{equation}
\label{boundarylogH}
 2 \log |H_{k}^{(\rho)}(\tau)| - \log |H_{k-1}^{(\rho)}(\tau)| - \log |H_{k+1}^{(\rho)}(\tau)|= 0 , \quad \tau\in[a_{k},b_{k}], \quad k = 0,\ldots,p-1.
\end{equation}
For each $0\leq k\leq p-1$, the function $u_k^{(\rho)}(z) = \log |H_{k}^{(\rho)}(z)|, k=0,\ldots,p-1$ is harmonic in $\overline{\mathbb{C}} \setminus [a_k,b_k]$. It is the real part of $\log H_k^{(\rho)},k=0,\ldots,p-1$ which is holomorphic in $\overline{\mathbb{C}} \setminus [a_k,b_k]$. Thus, $\log H_k^{(\rho)}$ is uniquely determined by $u_k$, taking into consideration that $H_k^{(\rho)}(\infty) > 0$ so we must find the harmonic conjugate $v_k^{(\rho)}$ of $u_k^{(\rho)}$ which equals zero at $\infty$. According to \cite[Lemma 4.1]{AptLopRocha},  the system \eqref{boundarylogH} only has the trivial solution (in the space of harmonic functions); that is, $u_k^{(\rho)} \equiv 0, k=0,\ldots,p-1.$ Consequently, $v_k^{(\rho)} \equiv 0, k=0,\ldots,p-1$ and $H^{(\rho)}_k= \exp(u^{(\rho)}_k + iv^{(\rho)}_k) \equiv 1, k=0,\ldots,p-1$. This means that $F^{(\rho)}_k \equiv G^{(\rho)}_k, k=0,\ldots,p-1$ which is what we needed to prove.
\end{proof}

\subsection{Proof of Theorem~\ref{ratioP}}

To prove the existence of limit in \eqref{limitP} it is sufficient to show that \eqref{eq:limratios} does not depend on $\Lambda$. The collection of functions $(\widetilde{F}_k^{(\rho)}), k=0,\ldots,p-1$ verifying \eqref{eq:limratios} gives rise to a collection $(F_k^{(\rho)}), k=0,\ldots,p-1$ which satisfies a system of boundary value equations as indicated in Lemma \ref{boundary} and fulfills i)-iii).

Now, according to Lemma \ref{unique} there is only one collection of functions $(F_k^{(\rho)}), k=0,\ldots,p-1$ with these properties. Since $\widetilde{F}_k^{(\rho)}$ (which must have leading coefficient equal to one) is obtained dividing $F_k^{(\rho)}$ by its leading coefficient the proof of \eqref{limitP} is settled.

Notice that
\[\frac{P_{n + p(p+1) ,k}(z)}{P_{n,k}(z)} = \prod_{\rho =0}^{p(p+1) -1} \frac{P_{n + \rho +1 ,k}(z)}{P_{n + \rho,k}(z)}. \]
On the right hand side, we have representatives of  ratios of consecutive polynomials $P_{n,k}$ for all the residue classes of $n$ modulo $p(p+1)$. Fix $\bar{\rho} \in [0:p(p+1) -1]$, using \eqref{limitP} it follows that
\[\lim_{n = \bar\rho \mod p(p+1)} \frac{P_{n + p(p+1) ,k}(z)}{P_{n,k}(z)} = \prod_{\rho =0}^{p(p+1)-1}\widetilde{F}_{k}^{(\rho)}(z),\qquad z\in\mathbb{C}\setminus[a_{k},b_{k}],\]
where the right hand does not depend on $\bar{\rho} \in [0:p(p+1) -1]$.  Therefore, \eqref{limitP2} takes place.

\subsection{Proof of Corollary~\ref{cor:ratioQnan}}

Let $\rho\in[0:p(p+1)-1]$ be fixed. Replacing $z$ by $z^{p+1}$ in \eqref{limitP} for $k=0$, we obtain
\begin{equation}\label{eq:auxratioPn0}
\lim_{\lambda\rightarrow\infty}\frac{P_{\lambda p(p+1)+\rho+1,0}(z^{p+1})}{P_{\lambda p(p+1)+\rho,0}(z^{p+1})}=\widetilde{F}_{0}^{(\rho)}(z^{p+1}),\qquad z\in\mathbb{C}\setminus\Gamma_{0}.
\end{equation}
According to \eqref{eq:decompQn} we have $Q_{n}(z)=z^{\ell} P_{n,0}(z^{p+1})$ if $n \equiv \ell \mod (p+1)$, $0\leq \ell\leq p$ (see also the lines that follow Definiton~\ref{Def:polyPnk}). Hence from \eqref{eq:auxratioPn0} we deduce that if $\rho\equiv \ell \mod (p+1)$ and $\ell\in[0:p-1]$, then
\[
\lim_{\lambda\rightarrow\infty}\frac{z^{\ell+1}P_{\lambda p(p+1)+\rho+1,0}(z^{p+1})}{z^{\ell}P_{\lambda p(p+1)+\rho,0}(z^{p+1})}=\lim_{\lambda\rightarrow 0}\frac{Q_{\lambda p(p+1)+\rho+1}(z)}{Q_{\lambda p(p+1)+\rho}(z)}=z\widetilde{F}_{0}^{(\rho)}(z^{p+1}),\qquad z\in\mathbb{C}\setminus\Gamma_{0},
\]
and if $\rho\equiv p \mod (p+1)$, then
\[
\lim_{\lambda\rightarrow\infty}\frac{P_{\lambda p(p+1)+\rho+1,0}(z^{p+1})}{z^{p} P_{\lambda p(p+1)+\rho,0}(z^{p+1})}=\lim_{\lambda\rightarrow 0}\frac{Q_{\lambda p(p+1)+\rho+1}(z)}{Q_{\lambda p(p+1)+\rho}(z)}=\frac{\widetilde{F}_{0}^{(\rho)}(z^{p+1})}{z^{p}},\qquad z\in\mathbb{C}\setminus(\Gamma_{0}\cup\{0\}).
\]
This justifies \eqref{eq:ratioQn:1}--\eqref{eq:ratioQn:2}.

In the following argument we use the cyclic notation $\widetilde{F}_{0}^{(i)}\equiv\widetilde{F}_{0}^{(i+p(p+1))}$ for any $i$. Applying \eqref{threetermrecsecondkindpequena1} for $k=0$ we get that for $\rho\not\equiv p \mod (p+1)$,
\[
a_{\lambda p(p+1)+\rho}=\frac{P_{\lambda p(p+1)+\rho,0}(z)}{P_{\lambda p(p+1)+\rho-p,0}(z)}-\frac{P_{\lambda p(p+1)+\rho+1,0}(z)}{P_{\lambda p(p+1)+\rho-p,0}(z)},\qquad z\in\mathbb{C}\setminus [a_{0},b_{0}].
\]
Letting $\lambda\rightarrow\infty$ we obtain
\[
\lim_{\lambda\rightarrow\infty} a_{\lambda p(p+1)+\rho}=a^{(\rho)}=(1-\widetilde{F}_{0}^{(\rho)}(z)) \prod_{i=\rho-p}^{\rho-1} \widetilde{F}_{0}^{(i)}(z),\qquad z\in\mathbb{C}\setminus[a_{0},b_{0}].
\]
Similarly, using \eqref{threetermrecsecondkindpequena2} we obtain that for $\rho\equiv p \mod (p+1)$,
\[
\lim_{\lambda\rightarrow\infty} a_{\lambda p(p+1)+\rho}=a^{(\rho)}=(z-\widetilde{F}_{0}^{(\rho)}(z)) \prod_{i=\rho-p}^{\rho-1} \widetilde{F}_{0}^{(i)}(z),\qquad z\in\mathbb{C}\setminus[a_{0},b_{0}].
\]
Hence \eqref{eq:limreccoeff} is justified.

Since
\[
\deg(P_{n+1,0})-\deg(P_{n,0})=Z(n+1,0)-Z(n,0)=\begin{cases}
0, & \mbox{if} \ n\not\equiv p \mod (p+1),\\
1, & \mbox{if} \ n\equiv p \mod (p+1),
\end{cases}
\]
from \eqref{limitP} we deduce that $\widetilde{F}_{0}^{(\rho)}$ has the Laurent expansion
\[
\widetilde{F}_{0}^{(\rho)}(z)=\begin{cases}
1+C_{1}^{(\rho)} z^{-1}+O(z^{-2}), & \mbox{if} \ \rho\not\equiv p \mod (p+1)\\
z+C_{0}^{(\rho)}+O(z^{-1}), & \mbox{if} \ \rho\equiv p \mod (p+1)
\end{cases}
\]
at infinity. Therefore, for $\rho\not\equiv p \mod (p+1)$, we have
\begin{align*}
a^{(\rho)}=(1-\widetilde{F}_{0}^{(\rho)}(z)) \prod_{i=\rho-p}^{\rho-1} \widetilde{F}_{0}^{(i)}(z)
=(-C_{1}^{(\rho)} z^{-1}+O(z^{-2}))(z+O(1))=-C_{1}^{(\rho)},
\end{align*}
for among the functions in $\prod_{i=\rho-p}^{\rho-1} \widetilde{F}_{0}^{(i)}(z)$ there is exactly one of the form $z+O(1)$ and the rest are of the form $1+O(z^{-1})$. Similarly one shows that if $\rho\equiv p \mod (p+1)$ then $a^{(\rho)}=-C_{0}^{(\rho)}$.

\subsection{Proof of Corollary~\ref{ratioP3}}

By Theorem \ref{ratioP}, the limits in \eqref{eq:asympgnk:1}, \eqref{eq:asympgnk:3}, and \eqref{eq:asympgnk:2} hold as $\lambda \to \infty$. Reasoning as in Subsection \ref{preliminary}, but now in connection with orthonormal polynomials, from Theorems 1 and 2 in \cite{BarCalLop} it follows that
\begin{equation}\label{limitP4+}
\lim_{\lambda \to \infty}\frac{p_{\lambda p(p+1)+\rho+1,k}(z)}{p_{\lambda p(p+1)+\rho,k}(z)}= \left(w_k^{(\rho)}\right)^{1/2} \widetilde{F}_{k}^{(\rho)}(z),\qquad z\in\mathbb{C}\setminus[a_{k},b_{k}], \qquad k=0,\ldots,p-1.
\end{equation}
The numbers $w_k^{(\rho)}, k=0,\ldots,p-1$ are defined as in Subsection \ref{preliminary} but now we know that they will not depend on the subsequence of indices $\Lambda$. Dividing \eqref{limitP4+} by \eqref{limitP}, it follows that
\[
\lim_{\lambda \to \infty}\frac{\kappa_{\lambda p(p+1)+\rho+1,k}}{\kappa_{\lambda p(p+1)+\rho,k}}= \left(w_k^{(\rho)}\right)^{1/2}, \qquad k=0,\ldots,p-1.
\]
Now, \eqref{limitP3}, \eqref{limitP4}, and \eqref{kappa} are a consequence of \eqref{eq:cs}.

From the definition of $\kappa_{\lambda p(p+1)+\rho,k}$ we have
\[K_{\lambda p(p+1)+\rho,k} = \kappa_{\lambda p(p+1)+\rho,0}\cdots \kappa_{\lambda p(p+1)+\rho,k}.\]
Therefore, \eqref{limitP3+} follows directly from \eqref{limitP3}.

Using formulas \eqref{eq:def:Hnk} and \eqref{eq:def:hnk} for two consecutive indices $n, n+1$, we have
\[\frac{\psi_{n+1,k}(z)}{\psi_{n,k}(z)} =\frac{K_{n,k-1}^2}{K_{n+1,k-1}^2}\frac{P_{n,k-1}(z)}{P_{n+1,k-1}(z)} \frac{P_{n+1,k}(z)}{P_{n,k}(z)} \frac{h_{n+1,k}(z)}{h_{n,k}(z)}.
\]
Taking $n=\lambda p(p+1)+\rho$ with $\rho$ fixed and $\lambda\rightarrow\infty$, from \eqref{limitP} and \eqref{limitP3+} we obtain
\[
\frac{K_{n,k-1}^2}{K_{n+1,k-1}^2}\frac{P_{n,k-1}(z)}{P_{n+1,k-1}(z)} \frac{P_{n+1,k}(z)}{P_{n,k}(z)}\longrightarrow \frac{1}{(\kappa_{0}^{(\rho)}\cdots\kappa_{k-1}^{(\rho)})^{2}}\frac{\widetilde{F}_{k}^{(\rho)}(z)}{\widetilde{F}_{k-1}^{(\rho)}(z)},
\]
uniformly on compact subsets of $\mathbb{C}\setminus([a_{k-1},b_{k-1}]\cup[a_{k},b_{k}])$. Now we write
\[
\frac{h_{n+1,k}(z)}{h_{n,k}(z)}=\frac{\varepsilon^{2}_{n+1,k-1}\, h_{n+1,k}(z)}{\varepsilon^{2}_{n,k-1}\, h_{n,k}(z)}
\]
and analyze the expressions
\[
\frac{\varepsilon_{n+1,k-1}}{\varepsilon_{n,k-1}},\qquad \frac{\varepsilon_{n+1,k-1}\, h_{n+1,k}(z)}{\varepsilon_{n,k-1}\, h_{n,k}(z)},
\]
separately.

Taking $n=\lambda p(p+1)+\rho$ with $\rho$ fixed, $\lambda\rightarrow\infty$, and applying \eqref{converghnk}, we easily see that
\[
\frac{\varepsilon_{\lambda p(p+1)+\rho+1,k-1}\, h_{\lambda p(p+1)+\rho+1,k}(z)}{\varepsilon_{\lambda p(p+1)+\rho,k-1}\, h_{\lambda p(p+1)+\rho,k}(z)}\longrightarrow h_{k}^{(\rho)}(z),
\]
uniformly on compact subsets of $\mathbb{C}\setminus([a_{k-1},b_{k-1}]\cup\{0\})$, where $h_{k}^{(\rho)}(z)$ is defined in \eqref{def:hkrho}.

We claim that the ratio
\begin{equation}\label{expressionepskr}
\frac{\varepsilon_{\lambda p(p+1)+\rho+1,k-1}}{\varepsilon_{\lambda p(p+1)+\rho,k-1}}=\varepsilon_{k}^{(\rho)}
\end{equation}
is independent of $\lambda$, and so it remains constant as $\lambda\rightarrow\infty$. This follows immediately from Lemma~ \ref{lem:ratioepsilon} and Lemma~\ref{lem:valuesdiffZnk}. With this we finish the proof of \eqref{second}.

Finally, the asymptotic formula \eqref{asymp:Psink} is obtained immediately from \eqref{second} and \eqref{modifiedPsink}. We leave the details to the reader.

\begin{remark}\label{rmksignepsilon}
We can give an explicit formula for the constant $\varepsilon_{k}^{(\rho)}$ in \eqref{second}. For a given $\rho\in[0:p(p+1)-1]$, let $\eta$, $\ell$ be the integers satisfying \eqref{eq:decomp:rho}. Let $\zeta(\rho,k)$ and $\theta(\rho,k)$ be the functions defined as follows:
\begin{align*}
\zeta(\rho,k) & :=\begin{cases}
\mbox{RHS of} \ \eqref{eq:values:diffZnk:1} & \mbox{if}\,\,\ell\in[0:p-1],\\
\mbox{RHS of} \ \eqref{eq:values:diffZnk:2} & \mbox{if}\,\,\ell=p,
\end{cases}\\
\theta(\rho,k) & :=\mbox{RHS of} \ \eqref{def:thetank}.\\
\end{align*}
Then from \eqref{expressionepskr} and \eqref{eq:ratioepsilon} we deduce that
\[
\varepsilon_{k}^{(\rho)}=(-1)^{\zeta(\rho, 2\lc \frac{k-1}{2}\rc)+\theta(\rho,k-1)},\qquad k=1,\ldots,p.
\]
\end{remark}

\end{document}